\documentclass[preprint,1p]{elsarticle}

\makeatletter
 \def\ps@pprintTitle{%
 	\let\@oddhead\@empty
 	\let\@evenhead\@empty
 	\def\@oddfoot{\footnotesize\itshape
 		{} \hfill\today}%
 	\let\@evenfoot\@oddfoot
 }
\makeatother
\usepackage{lipsum} 
\ExplSyntaxOn
\NewDocumentEnvironment{multiequation}{b}
 {
  \vantiempham:n { #1 }
 }
 {}

\seq_new:N \l__vantiempham_md_seq

\cs_new_protected:Nn \vantiempham:n
 {
  \seq_set_split:Nnn \l__vantiempham_md_seq { \\ } { #1 }
  $$
  \seq_map_function:NN \l__vantiempham_md_seq \__vantiempham_md_item:n
  $$
 }
\cs_new_protected:Nn \__vantiempham_md_item:n
 {
  \begin{minipage}{\dim_eval:n {\displaywidth/(\seq_count:N \l__vantiempham_md_seq)} }
  \begin{equation}
  \cs_set_eq:Nc \label { ltx@label }
  #1
  \vphantom{\seq_use:Nn \l__vantiempham_md_seq {}}
  \end{equation}
  \end{minipage}
 }

\ExplSyntaxOff  
\usepackage[unicode]{hyperref}
\usepackage{enumitem}
\makeatletter
\providecommand*{\cupdot}{%
  \mathbin{%
    \mathpalette\@cupdot{}%
  }%
}
\newcommand*{\@cupdot}[2]{%
  \ooalign{%
    $\m@th#1\cup$\cr
    \hidewidth$\m@th#1\cdot$\hidewidth
  }%
}
\makeatother
\usepackage{circuitikz}

\usepackage{tikz-cd}
\usepackage{multirow}
\usepackage{url}
\usepackage{pst-node}
\usepackage{comment}

\usepackage[makeroom]{cancel}
 
\usepackage{soul}
\usepackage{tikz}

\usepackage{latexsym}
\usepackage{indentfirst}
\usepackage{amsxtra}
\usepackage{amssymb}
\usepackage{amsthm}
\usepackage{amsmath}
\usepackage{mathrsfs} 

\usepackage{xcolor}
\usepackage{color}

\usepackage{amsfonts}

\usepackage[capitalise]{cleveref}

\newtheorem{theor}{Theorem}[section]
\newtheorem{prop}[theor]{Proposition}
\newtheorem{lemma}[theor]{Lemma}
\newtheorem{cor}[theor]{Corollary}
\theoremstyle{definition}               
\newtheorem{defin}[theor]{Definition}
\newtheorem{ex}[theor]{Example}
\newtheorem{exs}[theor]{Examples}
\newtheorem{rem}[theor]{Remark}

\newtheorem*{alg*}{Algorithm}

\DeclareMathOperator{\Sym}{Sym}

\DeclareMathOperator{\Aut}{Aut}
\DeclareMathOperator{\End}{End}

\DeclareMathOperator{\id}{id}

\DeclareMathOperator{\E}{E}

\newcommand{\rhoo}[2]{\rho_{#1}{#2}}
\newcommand{\sigmaa}[2]{\sigma_{#1}{#2}}
\newcommand{\lambdaa}[2]{\lambda_{#1}{#2}}

\newcommand{\alphaa}[3]{\alpha^{#1}_{#2}{#3}}
\newcommand{\betaa}[3]{\beta^{#1}_{#2}{#3}}

\newcommand{\phii}[2]{\phi_{#1,#2}}

\usepackage{lastpage}

\makeatletter
\def \@evenhead {\thepage\ of \pageref{LastPage} \hfil \slshape \leftmark } 
\def \@oddhead {{\slshape \rightmark }\hfil \thepage\ of \pageref{LastPage}} 
\makeatother

\begin{document}

\let\today\relax 

\begin{frontmatter}

 	\title{Set-theoretic solutions of the Yang--Baxter equation from inverse braces}
	\tnotetext[mytitlenote]{This work was partially supported by the Department of Mathematics and Physics ``Ennio De Giorgi", University of Salento. 
  The authors are all members of GNSAGA (INdAM). \\
F.Catino and M.Mazzotta are members of the nonprofit association ``AGTA-Advances in Group Theory and Applications".}

    \author {Francesco~CATINO}
  \ead{francesco.catino@unisalento.it}
	\author {Marzia~MAZZOTTA}
  \ead{marzia.mazzotta@unisalento.it}
	\author{Susanna~TIENI}
	\ead{susanna.tieni@unisalento.it}
\address{Department of Mathemathics and Physics ``Ennio De Giorgi", University of Salento\\
	Via Provinciale Lecce-Arnesano, 73100, Lecce (Italy)}

	\begin{abstract}
We introduce the algebraic structure of an \emph{inverse brace}, namely, a triple $(S,+,\circ)$ such that both $(S,+)$ and $(S,\circ)$ are inverse semigroups and the following identity holds
$a\circ(b+c)=a\circ b-a+a\circ c$,
for all $ a, b, c \in S$, where $-a$ denotes the inverse of $a \in S$, with respect to $+$. In particular, every weak brace is an inverse brace. We investigate the fundamental properties of inverse braces and analyze the relationship between additive and multiplicative idempotents, characterizing the condition under which an inverse brace is a weak brace. Our main results concern the connection with set-theoretic solutions to the Yang--Baxter equation. Specifically, we provide a class of inverse braces that yield solutions and give several examples.
Finally, we introduce constructions of inverse braces via the \emph{matched product} and the \emph{strong semilattice} of inverse braces. We show that these constructions preserve the conditions required to produce solutions, thereby providing a systematic method for generating new examples.

\end{abstract} 
	

\end{frontmatter}


\section*{Introduction}
The quantum Yang--Baxter equation is a fundamental object in mathematical physics, first introduced in \cite{Ya67} and, independently, in \cite{Ba72}. A key development came from Drinfel’d \cite{Dr92}, who posed the problem of classifying all set-theoretic solutions to this equation. Finding all such solutions up to equivalence is highly non-trivial and has motivated extensive research over the years. 
Let $S$ be a set. A map $r:S\times S \to S\times S$ satisfying the relation
\begin{align*}
    \left(r \times \id_S\right) 
    \left(\id_S \times r\right)
    \left(r \times \id_S\right) 
    = \left(\id_S \times r\right)
    \left(r \times \id_S\right)
    \left(\id_S \times r\right)
\end{align*}
is said to be \emph{set-theoretic solution of the Yang-Baxter equation}, or, shortly, \emph{solution}, on $S$.  We write $r\left(a, b\right) = \left(\lambda_{a}\left(b\right), \rho_{b}\left(a\right)\right)$, where $\lambda_{a}$ and $\rho_{b}$ are maps from $S$ into itself, for all $a, b\in S$.  
With this notation, we say that a solution $r$ is \textit{bijective} if $r$ is a bijective map and, in particular, \textit{involutive} if $r^2=\id_{S\times S}$; 
\textit{left non-degenerate} if $\lambda_a \in \Sym_S$, for all $a \in S$; 
\textit{right non-degenerate} if $\rho_b \in \Sym_S$, for all $b \in S$; \emph{non-degenerate} if it is both left and right non-degenerate. It has only recently been proven in \cite{JePi25} that all non-degenerate solutions are bijective.
Moreover, recently the notions of \emph{quasi bijectivity} and \emph{quasi non-degeneracy} have been introduced in \cite[Definition 3.1, Definition 3.3]{MaSteWi26x}, meaning that the maps $r$, $\lambda_a$, and $\rho_b$ are completely regular maps of the corresponding semigroups.

Braces were introduced by Rump in 2007 \cite{Ru07} as an algebraic tool for studying involutive non-degenerate solutions. More precisely, a \emph{(left) brace} is a set $B$ equipped with two binary operations, usually denoted by $+$ and $\circ$, such that $(B, +)$ is an abelian group, $(B, \circ)$ is a group, and
\begin{align*}
    a \circ (b+c)=a \circ b-a+a\circ c,
\end{align*}
for all $a, b, c \in B$, where $-a$ denotes the inverse of $a$ in the additive group $(B, +)$. Subsequently, Guarnieri and Vendramin \cite{GuVe17} introduced the notion of \emph{skew (left) brace}, generalizing braces by dropping the requirement that the additive group $(B,+)$ be abelian. Any skew left brace yields non-degenerate solutions.

Starting from the seminal work of Rump, a variety of brace-like structures have been developed, leading to significant advances in the understanding of non-degenerate, one-sided non-degenerate, and degenerate solutions. More recently, weak braces \cite{CaMaMiSt22} have been introduced as a natural extension of skew braces, drawing from the theory of inverse semigroups. Recall that a semigroup $S$ is an \emph{inverse semigroup} if every $a \in S$ admits a unique inverse $a^{-1}\in S$ such that $aa^{-1}a=a$ and $a^{-1}aa^{-1}=a^{-1}$.
Equivalently, inverse semigroups are regular semigroups in which the idempotents commute. An inverse semigroup is called \emph{Clifford} if the idempotents are central. Standard references for these semigroups can be found in \cite{Ho95, La98, Law26}, the last two of which are monographs on inverse semigroups. A \emph{weak (left) brace} is a triple $\left(S,+,\circ\right)$ such that both $\left(S,+\right)$ and $\left(S,\circ\right)$ are inverse semigroups satisfying 
\begin{align*}
    a \circ (b+c)=a \circ b-a+a\circ c \quad \text{and} \quad a \circ a^-=-a+a,
\end{align*}
for all $a, b, c \in S$, where $-a$ and $a^-$ denote the inverses of $a$ with respect to $+$ and $\circ$, respectively. It follows directly from the definition that the sets of idempotents $\E(S,+)$ and $\E(S, \circ)$ coincide. Skew braces are examples of weak braces since, in this case, the two groups necessarily share the same identity.
Moreover, $(S,+)$ is always a Clifford semigroup (see \cite[Theorem 8]{CaMaMiSt22}). If $(S, \circ)$ is also a Clifford semigroup, then $S$ is called \emph{dual weak brace}. A class of dual weak braces can be obtained from Rota--Baxter operators on Clifford semigroups (see \cite{CaMaSt23}). Additionally, according to \cite[Theorem 2.1]{CaMaSt24}, any dual weak brace decomposes as a \emph{strong semilattice} of a family of skew braces $\{B_\alpha\}_{\alpha\in Y}$ indexed by a semilattice $Y$. Any weak brace $(S, +, \circ)$ gives rise to a solution $r_S:S\times S\to S\times S$ defined by
\begin{align}\label{2}
    r_S\left(a,b\right)
    = \left(\lambda_a(b), \ \lambda_a(b)^-\circ a\circ b\right),
\end{align}
where $\lambda_a(b)=-a+a\circ b$, for all $a,b\in S$. Such a solution $r_S$ is quasi bijective (see \cite[Theorem 11]{CaMaMiSt22}). Moreover, when $S$ is a dual weak brace, $r_S$ is also quasi non-degenerate (see \cite[Theorem 13]{CaMaMiSt22}). Note that, in any weak brace $S$, one has $\lambda_a(b)=a \circ \left(a^-+b\right)$, for all $a, b \in S$, thus the map $r_S$ in \eqref{2} can be expressed in different equivalent forms. Besides, the map $\lambda: S \to \End(S, +)$, $a \mapsto \lambda_a$ is a homomorphism from $(S, \circ)$ to $\End(S, +)$. The study of the structure of weak braces has been further developed in a series of papers, including \cite{CaMaSt23, GoWa26,  LiWa25, MaRySt25, MaSteWi26x, Wang25x}.

This paper aims to further develop these brace-like structures by introducing and studying a new algebraic structure. In particular, an \emph{inverse (left) brace} is a triple $\left(S,+,\circ\right)$ such that both $\left(S,+\right)$ and $\left(S,\circ\right)$ are inverse semigroups and \begin{align*}
    a \circ (b+c)=a \circ b-a+a\circ c,
\end{align*}
for all $a, b \in S$. We prove that, for every inverse brace $(S, +, \circ)$, the inclusion $\E(S, +) \subseteq \E(S, \circ)$ holds.
Moreover, $S$ is a weak brace if and only if the two sets of idempotents coincide. We show that many properties of weak braces are inherited by these structures, whereas some others appear in a slightly weaker form.\\
 The main results of the paper concern the connection between solutions and inverse braces. Unlike the case of weak braces, not every inverse brace gives rise to a solution, and the associated map cannot be expressed in equivalent forms as in the weak brace setting. This is a consequence of the fact that, in general, $a \circ \left(a^-+b\right) \neq -a+a \circ b$. In fact, they coincide if and only if $S$ is a weak brace. We denote by $\lambda_a(b)=-a+a\circ b$ and $\sigma_a(b)=a \circ \left(a^-+b\right)$, for all $a, b \in S$, to distinguish the two maps. Moreover, the map $\lambda: S \to \End(S, +)$, $a \mapsto \lambda_a$ is a homomorphism from $(S, \circ)$ to $\End(S, +)$ like in the case of weak braces, whereas the map $\sigma: S \to S^S$, $a \mapsto \sigma_a$ is not, in general. By the way, $\sigma_a \notin \End(S, +)$, in general. These observations lead us to investigate the class of inverse braces for which the map $\sigma$ is a homomorphism, equivalently, for all $a \in S$ and $e \in \E(S, \circ)$, the following identity holds
\begin{align}
  \sigma_a(e)= a \circ e \circ a^-.\tag{I}
\end{align}
We prove that, for these inverse braces, the map $r_S: S  \times S \to S \times S$ given by
\begin{align*}
    r_S\left(a,b\right)
    = \left(\sigma_a(b), \ \sigma_a(b)^-\circ a \circ b\right),
\end{align*}
is a solution if and only if
\begin{align}
 \sigma_a(b) \circ \sigma_a(b)^- =   a \circ b \circ (a \circ b)^-\tag{II}
\end{align}
for all $a, b \in S$. 
We also investigate the non-degeneracy properties of the resulting solutions and describe situations in which they are quasi non-degenerate or quasi bijective. In particular, we prove that if the additive semigroup of an inverse brace $S$ is commutative, the associated solution is cubic, that is, $r_S^3=r_S$.

We present several methods for constructing new inverse braces from known ones, which yield solutions. These methods are inspired by analogous constructions employed for other brace-like structures, where they have proved successful. In particular, we develop the \emph{matched product} and the \emph{strong semilattice} of inverse braces, proving that these constructions preserve the hypotheses ensuring the existence of solutions. As a consequence, they provide systematic methods for producing new classes of solutions from existing examples. 

Finally, we propose some open problems concerning these structures, suggesting directions for further research.

\medskip

\section{Preliminaries}
 To set the notation, we begin by recalling essential notions on inverse semigroups needed for
our treatment.  Standard references for these semigroups can be found in \cite{Ho95, La98, Law26}, the last two of which are monographs on inverse semigroups. 

If $S$ is a semigroup and $a \in S$, we say that an element $x$ of $S$ is an \emph{inverse} of $a$ if $axa=a$ and $xax=x$. The semigroup $S$ is called \emph{inverse} if, for any $a\in S$, there exists a unique inverse of $a$, which we denote by $a^{-1}$. Clearly, every group is an inverse semigroup.  The behaviour of inverse elements in an inverse semigroup is similar to that in a group, since $(ab)^{-1}=b^{-1}a^{-1}$ and $(a^{-1})^{-1}=a$, for all $a,b \in S$. If $a \in S$, then $aa^{-1}$ and $a^{-1} a$ are idempotents of $S$. Moreover, the set $\E(S)$ of the idempotents is a commutative subsemigroup of $S$ and $e=e^{-1}$, for every $e \in \E(S)$.\\
\noindent
Inverse semigroups in which idempotents are central are called \emph{Clifford semigroups}. In this case,  $aa^{-1}= a^{-1}a$, for all $a \in S$.\\
It is well-known that any inverse semigroup $S$ admits a natural partial order relation defined by $a \leq b$ if $a=aa^{-1}b$, for all $a,b \in S$ (see \cite[Section 5.2]{Ho95}). This order relation is compatible with both the multiplication and the inversion on $S$. If $S$ is a group, then $\leq$ clearly coincides with equality. If $S$ is a semilattice, then $\leq$ coincides with the natural partial order on the semilattice.


\medskip

Let us now recall the following definition contained in \cite[Definition 5]{CaMaMiSt22}.
\begin{defin} 
    Let $S$ be a set endowed with two operations $+$ and $\circ$ such that $\left(S,+\right)$ and $\left(S,\circ\right)$ are inverse semigroups. Then, $(S, +, \circ)$ is a \emph{ weak (left) brace} if
 \begin{align*}
    a\circ\left(b + c\right)
    = a\circ b -a+ a\circ c\qquad \text{and}\qquad a\circ a^-=-a+a,
 \end{align*}
for all $a,b,c\in S$, where $-a$ and $a^-$ denote the inverses of $a$ with respect to $+$ and $\circ$, respectively.  
\end{defin}

Clearly, in any weak brace $(S,+,\circ )$, the sets of idempotents $\E(S,+)$ and $\E(S,\circ)$ coincide, so we simply denote this set by $\E(S)$. 
In addition, by \cite[Lemma 1.4]{CaMaSt24}, one has that 
\begin{align}\label{eq:idem+circ}
   \forall\, a\in S, \ e \in \E(S) \qquad  e \circ a=e+a.
\end{align}
In particular, the operations of sum and multiplication coincide on $\E(S)$ (see also \cite{CaMaSt23}). Obviously, if $|\E(S)|=1$, then $(S,+, \circ)$ is a \emph{skew (left) brace} \cite{GuVe17}. In this case, $(S,+, \circ)$ is a \emph{(left) brace} if $(S, +)$ is abelian.
In \cite[Theorem 8]{CaMaMiSt22}, it is proved that the additive semigroup of any weak brace is necessarily Clifford. A weak brace $(S, +, \circ)$ is called \emph{dual weak brace} if also $(S, \circ)$ is Clifford. Note that, any Clifford semigroup $\left(S, \circ\right)$ determines two trivial dual weak braces, by setting $a + b:= a\circ b$ or $a + b:= b\circ a$, for all $a, b\in S$.




\medskip

The motivation for studying weak braces is that each weak brace naturally yields a solution, as we recall below. To do this, we introduce the map $\lambda: S\to \End(S,+), \,a\mapsto\lambda_a$ defined by $$\lambda_a\left(b\right) = -a+a\circ b,$$
for all $a, b\in S$. It is known from \cite[Proposition 7]{CaMaMiSt22} that the map $\lambda$ is a homomorphism of $\left(S,\circ\right)$ into the endomorphism semigroup of $\left(S,+\right)$. 

\begin{theor}\emph{(\cite[Theorem 11]{CaMaMiSt22})}\label{teo_solu_weak}
    Let $(S, +, \circ)$ be a weak brace and $r_S:S\times S\to S\times S$ the map defined by $$r_S\left(a,b\right)
    = \left(\lambda_a(b), \ \lambda_a(b)^- \circ a \circ b\right),$$ 
for all $a,b\in S$. Then $r_S$ is a solution.
\end{theor}

According to \cite[Theorem 13]{CaMaMiSt22}, the solution $r_S$ associated with a weak brace $S$ is \emph{quasi bijective}. We recall the general definition provided in \cite[Definition 3.1]{MaSteWi26x}.

\begin{defin}\label{defn-quasi}
      If $X$ is a set, a solution $r:X \times X  \to X \times X$ is said to be \emph{quasi bijective} if there exists a (unique) solution $\left(X, r^{-}\right)$ such that
      \begin{align*}
          rr^-r=r, \quad r^-rr^-=r^- \quad \text{and} \quad r^{-}r=rr^{-}.
      \end{align*}
\end{defin}

In the case of a weak brace $(S, +, \circ)$, $r_S^-$ coincides with the solution $r_{S^{op}}$ associated with the \emph{opposite weak brace} $S^{op}=\left(S, +^{op}, \circ\right)$ of $S$. If $S$ is also a dual weak brace, the solution $r_S$ is also \emph{quasi non-degenerate} (cfr. \cite[Definition 3.3]{MaSteWi26x} and \cite[Proposition 2.3]{MaSteWi26x}).

\begin{defin}\label{defn-left-right-quasi}
If $X$ is a set, a solution $r:X \times X  \to X \times X$ defined by $r(x,y)=(\lambda_x(y), \rho_y(x))$, for all $(x, y) \in X \times X$, is said to be:
    \begin{enumerate}
        \item[1)] \emph{quasi left non-degenerate} if for all $x \in X$ there exists a (unique) map $\lambda_x^{-}: X \to X$ such that, for all $x,y \in X$,
   \begin{align*}
       \lambda_x\lambda_x^{-}\lambda_x=\lambda_x, \quad \lambda_x^-\lambda_x\lambda_x^-=\lambda_x^-, \quad \lambda_x^0:=\lambda_x\lambda_x^{-}=\lambda_x^-\lambda_x, \quad\text{and}\quad \lambda_x^0\lambda_y=\lambda_y\lambda_x^0;
    \end{align*}
    \item[2)] \emph{quasi right non-degenerate} if for all $x \in X$
    there exists a (unique) map $\rho_x^{-}: X \to X$ such that, for all $x,y \in X$,
    \begin{align*}
       \rho_x\rho_x^{-}\rho_x=\rho_x, \quad \rho_x^-\rho_x\rho_x^-=\rho_x^-, \quad  \rho_x^0:=\rho_x\rho_x^{-}=\rho_x^-\rho_x, \quad\text{and}\quad \rho_x^0\rho_y=\rho_y\rho_x^0;
    \end{align*}
    \item[3)] \emph{quasi non-degenerate} if $r$ is both quasi right and left non-degenerate.
    \end{enumerate}
\end{defin}

\smallskip

\noindent In particular,  if $S$ is a dual weak brace, then the solution $r_S$ is quasi bijective and quasi non-degenerate with $\lambda_a^-=\lambda_{a^-}$ and $\rho_b^-=\rho_{b^-}$, for all $a, b \in S$, where $\rho_b(a)=\lambda_a(b)^- \circ a \circ b$.
    
\medskip

\section{Basic properties of inverse braces}
 In this section, we introduce the algebraic structure of inverse brace and give some basic properties that will be helpful throughout the paper.

\smallskip

\begin{defin}
    A triple $(S, +, \circ)$ is an \emph{inverse (left) brace} if $(S, +)$ and $(S, \circ)$ are inverse semigroups and 
    \begin{equation*}\label{distrib_law}
        a \circ (b+c)=a\circ b -a+a \circ c
    \end{equation*}
    holds, for all $a,b,c \in S,$ where $-a$ denotes the inverse of $a$ in $(S, +)$. Usually, we denote by $a^-$ the inverse of $a$ in $(S,\circ)$.
\end{defin}

\smallskip
 We observe that every weak brace is clearly an inverse brace. The following inverse braces are not weak braces.

\begin{ex}\label{esempiowang}
    Consider the additive group  $\mathbb{Z}_2=\{0,1\}$ of integers modulo $2$ equipped with the operation $\circ$ defined by setting $ 0\circ 0=0, 1\circ 1=1\circ 0=0\circ 1=1$. Then $\left(\mathbb{Z}_2,+, \circ\right)$ is an inverse brace that is not a weak brace. Indeed, $1\circ 1^-=1\circ 1=1$, while $-1+1=0$.
   \end{ex}

More generally, we can give the following broad class of examples.
\begin{ex}\label{ex_ring}
Let $(R, +,\cdot)$ be a ring and define the adjoint operation $\circ$ on $R$ by setting $a \circ b := a + ab + b,$ for all $a, b \in R$. 
If $(R,\circ)$ is an inverse semigroup, then it is straightforward to verify that $(R, +, \circ)$ is an inverse brace. 
One can see that $(R,\circ)$ is inverse if and only if $R$ is an extension of a strongly regular ring by a radical ring  (see \cite[Theorem 6]{Xi88}). Recall that a ring $(R,+,\cdot)$ is \emph{strongly regular} if $a\in a^2 R$, for all $a\in R$.
\end{ex}

\medskip
As in the theory of weak braces, if $(S, +, \circ)$ is an inverse brace, let us consider for all $a \in S$ the map
	$$
	\lambda_a:S\rightarrow S, \;\; b\mapsto -a + a\circ b.
	$$
The following lemmas can be proved as in \cite{CaMaMiSt22} for weak braces, since the additional condition on the idempotents is not used in the proofs. For this reason, we omit the proofs.

In what follows, let 
$\End(S,+)$ denote the semigroup of endomorphisms of the semigroup $(S,+)$.

\begin{lemma}\label{key2.4}
Let $(S,+,\circ)$ be an inverse brace. Then, for every $a\in S$, $\lambda_a \in \End(S,+)$. Moreover, the map $\lambda: S \rightarrow \End(S,+),\; a \mapsto \lambda_a$ is a homomorphism from the semigroup $(S,\circ)$ to $\End(S,+)$.
\end{lemma}

\smallskip

\begin{lemma}\label{lemma_ab}
Let $(S,+,\circ)$ be an inverse brace. Then, for all $a,b\in S$, the following identity hold:
\begin{enumerate}
\item[\rm{(1)}] $a\circ b= a - a\circ (-b)+a$,
\item[\rm{(2)}] $a\circ b = a+\lambda_a(b)$.
\end{enumerate}
\end{lemma}

\smallskip 

The following result provides a characterization of inverse braces that also appears in \cite[Proposition 3.1]{Wang25x}. We provide a different proof.

\begin{prop}\label{prop_homo_lambda}
     Let $(S, +)$ and $(S, \circ)$ be two inverse semigroups. Then $(S, +, \circ)$ is an inverse brace if and only if the following hold:
    \begin{enumerate}
       \item[\rm{(1)}] for all $a \in S$, the map $\lambda_a: S \rightarrow S, \;\; b \mapsto -a + a \circ b$ is a semigroup endomorphism of $(S, +)$;
       \item[\rm{(2)}] the map $\lambda: S \rightarrow \End(S, +), \;\; a \mapsto \lambda_a$ is a semigroup homomorphism from $(S, \circ)$ to $\End(S, +)$.
    \end{enumerate}
\end{prop}
\begin{proof}
    If $(S, +, \circ)$ is an inverse brace, then conditions $(1)$ and $(2)$ follow by \cref{key2.4}.\\ Conversely, assume that $(1)$ and $(2)$ hold and let $a, b \in S$. Hence, we have
    \begin{align*}
    -a\circ b+a\circ b&=\lambda_{a\circ b}\left(b^- \circ b\right)=\lambda_a\lambda_b\left(b^- \circ b\right)=\lambda_a(-b+b)\\
    &=\lambda_a(-b)+\lambda_a(b)=-a+a\circ (-b)-a+a\circ b.
    \end{align*}
   It follows that
    $a\circ b=a\circ b-a\circ b+a\circ b=a\circ b-a+a\circ (-b)-a+a\circ b$ and
    \begin{align*}
    (-a+ a\circ (-b)-a)&+a\circ b+(-a+a\circ (-b)-a)\\
    &=\lambda_a(-b)+\lambda_a(b)+\lambda_a(-b)-a\\
    &=\lambda_a(-b+b-b)-a\\
    &=\lambda_a(-b)-a\\
    &=-a+a\circ (-b)-a.
    \end{align*}
    Hence, by the uniqueness of the inverse in the inverse semigroup $(S,+)$, it follows that $a\circ b=a-a\circ(-b)+a.$
    Besides, we have
   $$a+\lambda_a(b) =-\left(\lambda_a(-b)-a\right)= a-a\circ (-b)+a=a\circ b.$$
Therefore, for all $a,b, c \in S$, we get
    \begin{align*}
        a\circ(b+c)=&a+\lambda_a(b+c)= a+\lambda_a(b)+\lambda_a(c)= a\circ b + \lambda_a(c)= a\circ b-a+a\circ c,
        \end{align*}
       namely, $(S, +, \circ)$ is an inverse brace.
\end{proof}

\smallskip

Now, we aim to highlight the substantial differences between inverse braces and weak braces, particularly concerning the existing relation between the two sets of idempotents. 

\begin{theor}\label{prop_idemp}
    Let $(S,+,\circ)$ be an inverse brace. Then $\E(S, +) \subseteq \E(S, \circ)$.
\end{theor}
\begin{proof}
If $e \in \E(S,+)$, we have
\begin{align*}
    e&
=e\circ e^- \circ(e+e)=e\circ(e^-\circ e -e^-+e^-\circ e)\\
    &= e\circ (e^-\circ e-e^-) -e +e\circ e^-\circ e\\
    &=e-e+e\circ(-e^-)-e+e\\
    &=-e+e\circ(-e^-)-e\\
    &= - e\circ e^-,
\end{align*}
where we use \cref{lemma_ab}(1) in the last equality. It follows that $e=-e=e\circ e^-$ and so $e=e\circ e^-\circ e= e\circ e$.
 Therefore, $e \in \E(S,\circ),$ which proves the claim.
\end{proof}

\smallskip

    It follows immediately from \cref{prop_idemp} that if the multiplicative structure has a unique idempotent, then so does the additive structure. Consequently:
\begin{cor}\label{circ_gruppo}
Let $(S,+,\circ)$ be an inverse brace such that $(S,\circ)$ is a group. Then $(S, +, \circ)$ is a skew brace. 
\end{cor}

\smallskip

\begin{rem}
    Note that if $(S, +, \circ)$ is an inverse brace such that both $(S, +)$ and $(S, \circ)$ are inverse monoids, then they have the same identity element. Indeed, let $0$ be the identity of $(S,+)$ and $e$ the identity of $(S, \circ)$. Then
   $$0=e \circ 0=e \circ (0+0)=e \circ 0-e+e \circ 0=0-e+0=-e,$$
  which implies that $e=-0=0$. 
\end{rem}

\smallskip

As recalled in the first section, in any weak brace $(S,+,\circ)$, we have $\E(S,+)=\E(S,\circ)$. Therefore, it is natural to ask whether the converse holds for inverse braces. As we show below, the answer is affirmative. We begin by establishing the following lemma.

\begin{lemma}\label{lemma_1}
    Let $(S,+,\circ)$ be an inverse brace. Then, for all $a,b\in S$, the following identity hold: 
\begin{enumerate}
\item[\rm{(1)}] $a\circ b = a\circ b -a+a,$
\item[\rm{(2)}] $a= a -a \circ a^-+a\circ a^-$ .
\end{enumerate}
\end{lemma}  
\begin{proof}
For $a, b \in S$, by \cref{lemma_ab}, we have 
   \begin{align*}
  a\circ b=& a - a\circ (-b)+a = a-a+a\circ b-a+a\\
  &=a-a+a\circ b -a+a-a+a\\&=a\circ b-a+a.
  \end{align*}
Thus, we get
$
a = (a\circ a^-)\circ a = (a\circ a^-)\circ a - a\circ a^- + a\circ a^- = a - a\circ a^-+a\circ a^-.
$
 \end{proof}     
\medskip\medskip

\begin{theor}\label{teo_inverse_weak}
    Let $(S,+,\circ)$ be an inverse brace such that $\E(S,+)=\E(S,\circ).$ Then, $(S,+,\circ)$ is a weak brace.
\end{theor}
\begin{proof}
    To prove the claim, it is enough to prove that $-a+a=a\circ a^-$, for all $a \in S$.\\
 Let $a \in S$. By \cref{lemma_1}(1), and since  $\E(S, \circ) \subseteq \E(S, +)$, we have
      \begin{align*}
          a\circ a^-= a\circ a^- -a+a =-a+a+a\circ a^-= -a+a-a\circ a^- + a\circ a^- = -a+a,
      \end{align*}
where we use \cref{lemma_1}(2) in the last equality. Therefore, $(S, +, \circ)$ is a weak brace.
 \end{proof}

\smallskip

   Following \cite{Wang25x}, an inverse brace $(S, +, \circ)$ is called \emph{symmetric} if \begin{align*}
a + (b\circ c) = (a+b)\circ a^-\circ (a+c),
\end{align*}
for all $a,b,c \in S$. In \cite[Proposition 4.2]{Wang25x}, it is proved that every symmetric inverse brace is a dual weak brace. The same result follows immediately from \cref{teo_inverse_weak}, as we show below.
\begin{cor} Every symmetric inverse brace is a dual weak brace. \end{cor}
\begin{proof}
    If $(S,+,\circ)$ is a symmetric inverse brace, then $\E(S,+)\subseteq \E(S,\circ)$ by \cref{prop_idemp}. Moreover, since $(S,+,\circ)$ is symmetric, we also have $\E(S,\circ)\subseteq \E(S,+)$. Hence, the equality holds, and the claim follows from \cref{teo_inverse_weak}.
\end{proof}

\medskip

Finally, we note that there is a further link between the additive and multiplicative idempotents of an inverse brace. Given an inverse brace $(S,+,\circ)$, consider the natural partial order relation $\leq_+$ and $\leq_{\circ}$ on $(S, +)$ and $(S,\circ)$, respectively.

\begin{prop}\label{order}
     Let $(S,+,\circ)$ be an inverse brace. Then
     \begin{align*}
         \forall e_1, e_2 \in \E(S,+) \quad e_1 \leq_+ e_2 \Longrightarrow e_1 \leq_\circ e_2.
     \end{align*}
\end{prop}
\begin{proof}
Initially, for all $e \in \E(S, +)$ and $a\in S$,  observing that $e=e \circ e$ by \cref{prop_idemp},  applying \cref{lemma_ab}(2), we obtain
\begin{align*}
    e \circ a=e+\lambda_e(a) =e \circ e+\lambda_e (a)=e \circ (e+a).
\end{align*}
    Now, let $e_1, e_2 \in \E(S, +)$ be such that $e_1 \leq_+ e_2$. Then, by \cref{prop_idemp}, $$e_1\circ e_2 = e_1\circ(e_1+e_2)=e_1\circ e_1=e_1,$$ that is, $e_1 \leq_\circ e_2$.
\end{proof}

\smallskip

\begin{ex}
 Consider the semilattices $(S, +)$ and $(S, \circ)$ on the set $S=\{a, b\}$ with binary operations given by $a+b=a$ and $a \circ b=b$.
In light of \cref{order}, since $a \leq_+ b$ and $a \nleq_\circ b$, these two operations cannot define a weak brace.
\end{ex}

\medskip

\section{The \texorpdfstring{$\sigma$}{} map associated with an inverse brace}

In the following, for every inverse brace $(S, +, \circ)$ and $a \in S$, let us denote by $\sigma_a$ the map 
	$$
	\sigma_a:S\rightarrow S, \;\; b\mapsto a \circ \left(a^-+b\right).$$
In any weak brace $S$, we have 
\begin{center}
$\sigma_a(b) = a\circ a^- -a+a\circ b = -a+a-a+a\circ b = -a +a\circ b =\lambda_a(b)$,
\end{center} for all $a, b \in S$. 
In general, this identity does not hold in the case of inverse braces, as we show in the next example.
\begin{ex}\label{ex_cliff}
  Let $S=\{0, x, y\}$. Consider the commutative inverse monoids $(S, +)$ and $(S, \circ)$, both with identity $0$, whose binary operations are given by 
  $$\begin{tabular}{c | c c c c }
    $+$ & $0$ & $x$ & $y$  \\
    \hline
   $0$& $0$ & $x$ & $y$ \\
    $x$ & $x$ & $x$ & $y$ \\
    $y$& $y$ & $y$ & $x$ 
\end{tabular} \qquad \text{and} \qquad 
\begin{tabular}{c | c c c c }$\circ$ & $0$ & $x$ & $y$  \\
    \hline
   $0$& $0$ & $x$ & $y$ \\
    $x$ & $x$ & $x$ & $y$ \\
    $y$& $y$ & $y$ & $y$ 
\end{tabular}$$
  Then $(S, +, \circ)$ is an inverse brace. Moreover, for example, $\sigma_y(y) =y$, while $\lambda_y(y)=x$.
\end{ex}

\smallskip
    
The following result expresses the general relationship between the maps $\sigma_a$ and $\lambda_a$.

\begin{prop}\label{rem_sigma_lambda}
If $(S, +, \circ)$ is an inverse brace, then, for all $a, b \in S$,
$$\sigma_a(b)=a \circ a^- + \lambda_a(b)=a \circ a^- \circ \lambda_a(b).$$
Moreover, $\sigma_a=\lambda_a$, for every $a \in S$, if and only if $S$ is a weak brace.
    \begin{proof}
        Let $a, b \in S$. Then, by \cref{key2.4}, we have
        \begin{align*}
        \sigma_a(b)&=a \circ a^-+\lambda_a( b)=a \circ a^-+\lambda_{a \circ a^- \circ a}(b)=a \circ a^-+\lambda_{a \circ a^- }\lambda_a(b)=a \circ a^- \circ \lambda_a(b),
        \end{align*}
        where we apply \cref{lemma_ab}(2) in the last equality. \\
   Now, assume that $\sigma_a=\lambda_a$, for every $a \in S$. If $a \in S$, 
 we have
   \begin{align*}
      -a+a&=\lambda_a\left(a^- \circ a\right)\\
      &=\sigma_a\left(a^- \circ a\right)\\
      &=a \circ a^- + \lambda_a\left(a^- \circ a\right)\\
      &=a \circ a^--a+a\\
      &=a \circ a^- &\mbox{by \cref{lemma_1}(1)}
   \end{align*}
   Hence, $S$ is a weak brace. The converse implication is obvious.
    \end{proof}
\end{prop}

\smallskip

In contrast to the map $\lambda_a$, the map $\sigma_a$ need not be an endomorphism of the additive structure,  as demonstrated by the following example.
\begin{ex}
    Let $(S,+,\circ)$ be the inverse brace in \cref{ex_cliff}. 
    Then, for example, $\sigma_y \notin \End(S,+)$. Indeed, $\sigma_y(x+y)=\sigma_y(y)=y$, while $\sigma_y(x)+\sigma_y(y)=y+y=x$.
   \end{ex}

\smallskip

\noindent In some cases, we can characterize when $\sigma_a$ is an endomorphism of $(S, +)$. We first prove the following result.

\begin{lemma}\label{lemma_ameno}
    Let $(S, +, \circ)$ be an inverse brace. Then $\sigma_a\left(-a^-\right)=a$, for every $a \in S$.
\end{lemma} 
\begin{proof}
Applying \cref{lemma_ab} and \cref{lemma_1}(1), we obtain
\begin{align*}
    \sigma_a\left(-a^-\right)&
    =a\circ a^- -a+ a\circ (-a^-)= a\circ a^--a+a-a \circ a^-+a\\
    &=a\circ a^- -a\circ a^-+a\circ a^- \circ a= a\circ a^- +\lambda_{a\circ a^-}(a)\\  
    &=a \circ a^- \circ a=a,
\end{align*}
for all $a \in S$.
\end{proof}

\smallskip

\begin{prop}
    Let $(S, +, \circ)$ be an inverse brace such that $(S, +)$ is a Clifford semigroup. 
    Then, for all $a \in S$, $\sigma_a \in \End(S, +)$ if and only if $S$ is a weak brace.
\end{prop}
\begin{proof}
    Assume that $\sigma_a \in \End(S, +)$, for all $a \in S$. Then, for all $a,b,c \in S$, we have \begin{align*} \sigma_a(b)+\sigma_a(c)=\sigma_a(b+c)=a \circ \left(a^-+b\right) -a + a \circ c=\sigma_a(b)+\lambda_a(c).
    \end{align*}
Thus, taking $b=-a^-$ and using \cref{lemma_ameno}, we obtain $a +\sigma_a(c)= a+\lambda_a(c)=a\circ c$, where the last equality follows from \cref{lemma_ab}(2).
     Choosing $c=a^- \circ a$, we get $$a+\sigma_a\left(a^- \circ a\right)=a \circ a^- \circ a=a.$$  Equivalently, $a+a \circ a^--a+a =a$. Hence, by \cref{lemma_1}(1), we get
     $$a+a \circ a^-=a.$$ 
     This implies that $-a+a+a \circ a^-=-a+a$.
     Using \cref{lemma_1}(1) and the fact that $(S, +)$ is Clifford, we have $-a+a=a \circ a^-$. Therefore, $S$ is a weak brace.\\
     The converse implication is obvious.
\end{proof}




\smallskip

In general, unlike the lambda map, the map $\sigma: S \to S^S, a \mapsto \sigma_a$ is not a homomorphism from the inverse semigroup $(S, \circ)$ to the monoid $S^S$. 
\begin{ex}\label{nhomo}
        Consider the ring  $\left(\mathbb{Z}_3,+,\cdot \right)$ of integers modulo $3$.
        Then $\left(\mathbb{Z}_3,+, \circ\right)$ is an inverse brace belonging to the class of inverse braces presented in \cref{ex_ring}. 
In this example, the map $\sigma$ is not a homomorphism since, for example, 
$\sigma_{1 \circ 2}(2)=2$, while $\sigma_{1}\sigma_2(2)=1$.
\end{ex}

\smallskip

\noindent We can prove the following result.
\begin{prop} \label{Aequivsigma}
  Let $(S, +, \circ)$ be an inverse brace. Then, the map $\sigma: S \to S^S, a \mapsto \sigma_a$ is a homomorphism from the inverse semigroup $(S, \circ)$ into the monoid $S^S$
  if and only if the following is satisfied
   \begin{align}\label{A} 
    &\sigma_a(e)=a \circ e \circ a^-\tag{I}
     \end{align}
     for all $a \in S$ and $e \in \E(S,\circ)$.
\end{prop}
\begin{proof}
First, for all $e \in \E(S, \circ)$, we obtain
\begin{align*}
    \sigma_e(e)=e \circ \left(e+e\right)=e \circ e-e+e \circ e=e-e+e=e.
\end{align*}
Assume that $\sigma$ is a homomorphism from $(S,\circ)$ into the monoid $S^S$.  Hence, for all $a\in S$ and $e\in E(S,\circ)$, we have
\begin{align*}
    \sigma_a(e) &=\sigma_a\sigma_e(e)=\sigma_{a\circ e}(e)=a\circ e\circ (e\circ a^-+e)\\
                &= a\circ e\circ a^--a\circ e+ a\circ e=a\circ e\circ a^-,
\end{align*}
where the last equality follows from \cref{lemma_1}(1).\\  
Vice versa, assume that \eqref{A} is satisfied and let $a, b, c \in S$. Then, using \cref{rem_sigma_lambda} and \cref{prop_homo_lambda}, we get
       \begin{align*}
        \sigma_{a \circ b}(c)&=a \circ b \circ b^- \circ a^-+\lambda_{a \circ b}(c)=\sigma_a\left(b \circ b^-\right)+\lambda_a\lambda_b(c) \\
          &=a\circ a^-+\lambda_a\left(b \circ b^- + \lambda_b(c)\right)=\sigma_a\sigma_b(c),
     \end{align*}
   which is our claim.
\end{proof}

\smallskip

\begin{ex}\label{ex_semilattice}
    Let $(S,+,\circ)$ be an inverse brace with $(S, \circ)$ a semilattice (for example, the one in \cref{ex_cliff}). Then, by \cref{lemma_ab}(2), for all $a, e\in S$, we have 
    \begin{align*}
    \sigma_a(e)&=a \circ (a+e)=a \circ a-a+a \circ e=a+\lambda_a(e)=a \circ e=a \circ e  \circ a.
\end{align*}
  Thus, by \cref{Aequivsigma}, the map $\sigma$ is a homomorphism from $(S, \circ)$ to $S^S$.
\end{ex}

\smallskip

\begin{rem}\label{A*}
    If $(S, +, \circ)$ is an inverse brace satisfying \eqref{A}, we also have that $\sigma_a(b)=a \circ b \circ b^- \circ \left(a^-+b\right)$, for all $a,b \in S$. In fact, 
   \begin{align*}
       a \circ b \circ b^- \circ \left(a^-+b\right)&=a \circ b \circ b^- \circ a^--a \circ b \circ b^-+a \circ b\\
       &=\sigma_a\left(b \circ b^-\right)-a \circ b \circ b^-+a \circ b &\mbox{by \eqref{A}}\\
       &=a \circ a^--a+a \circ b \circ b^--a \circ b \circ b^-+a \circ b &\mbox{by \cref{rem_sigma_lambda}}\\
       &=a \circ a^--a+a \circ b &\mbox{by \cref{lemma_ab}(2)}\\
       &=\sigma_a(b).
   \end{align*}
\end{rem}

\medskip

\section{Inverse braces that give rise to solutions}
In this section, we describe inverse braces for which $\sigma$ is a homomorphism from $(S, \circ)$ into $S^S$
which give rise to solutions of the form
  \begin{align}\label{r_S}
     r_S(a,b)=\left(\sigma_a(b), \sigma_a(b)^-  \circ a \circ b\right),
 \end{align}
 for all $a,b \in S$. 
By \cref{rem_sigma_lambda}, the map $r_S$ can be also written as
\begin{align*}
    r_S(a, b)=\left(a \circ a^-+ \lambda_a(b), \, \lambda_a(b)^- \circ a \circ b\right),
\end{align*}
for all $a,b \in S$. This formulation makes the similarity with the solution in \cref{teo_solu_weak} particularly transparent. The only difference lies in the idempotent term $a \circ a^-$, which, in the case of weak braces, is absorbed by $\lambda_a(b)$, since $a \circ a^-=-a+a$.

\medskip

For any inverse brace $(S, +, \circ)$, let us introduce the map
$\rho: S \to S^S,\, b \mapsto \rho_b$ defined by $$\rho_b(a)=\sigma_a(b)^-  \circ a \circ b,$$
for all $a, b \in S$. Thus, $\rho_b(a)$ denotes the second component of the map $r_S$. \\
In the following, we prove an important connection between $\sigma_a$ and $\rho_b$, which will be useful in what follows.

\begin{lemma}\label{sigma_rho}
Let $(S,+,\circ)$ be an inverse brace such that \eqref{A} is satisfied. Then, for all $a, b \in S$, we have $\sigma_a(b)=a \circ b \circ \rho_b(a)^-.$
\end{lemma}
\begin{proof} 
  If $a,b \in S,$ we have
    \begin{align*}
        a\circ b \circ \rho_b(a)^-&=a\circ b \circ b^- \circ a^- \circ a \circ \left(a^- +b\right)\\
        &=a\circ b \circ b^- \circ  \left(a^- +b\right) \\
        &=a\circ b \circ b^- \circ a^- +\lambda_{a\circ b \circ b^-}\left(b\right) \\
        &=\sigma_{a\circ b \circ b^-}\left(b\right) &\mbox{by \cref{rem_sigma_lambda}}\\
        &=\sigma_{a}\sigma_{b \circ b^-}\left(b\right)  &\mbox{by \eqref{A}}\\
        &=\sigma_a\left(b \circ b^--b \circ b^-+b \circ b^- \circ b\right)\\
        &=\sigma_a(b) &\mbox{by \cref{lemma_ab}(2)}
    \end{align*}
 which is our claim.
\end{proof}

\smallskip

\begin{rem}
    Observe that if $(S, +, \circ)$ is an inverse brace such that $(S, +)$ is a commutative semigroup and satisfying \eqref{A}, then
        $\rho_b(a)=\left(a^-+b \right)^- \circ b$, for all $a, b \in S$.\\
    Indeed, for $a, b \in S$, we have
    \begin{align*}
        \rho_b(a)^-
        &=b^- \circ a^- -b^- \circ a^- \circ a + b^- \circ a^-  \circ a \circ b\\
        &=b^- \circ a^- -b^- \circ a^- \circ a +\sigma_{b^-}\left(a^- \circ a\right)\\
        &=b^- \circ a^-  - b^- \circ a^- \circ a+b^- \circ a^- \circ a -b^-+b^- \circ b\\
        &=b^- \circ a^- -b^-+b^- \circ b &\mbox{by \cref{lemma_1}(1)}\\
        &=b^- \circ \left(a^-+b\right).
    \end{align*}
\end{rem}

\smallskip

In the following result, we study the condition under which the map $\rho$ is an anti-homomorphism and the product of the two components of $r_S(a,b)$ in \eqref{r_S} coincides with $a \circ b$, as is the case for a weak brace. The condition imposed below will be necessary for $r_S$ to define a solution. 
\begin{prop}\label{condizione_a_2}
  Let $(S, +, \circ)$ be an inverse brace satisfying \eqref{A} and such that \begin{align}
    &\sigma_a(b)\circ \sigma_a(b)^-=\sigma_a\left(b \circ b^-\right) 
    \tag{II}\label{B}
    \end{align}
    holds, for all $a, b \in S$. Then, the following hold:
    \begin{enumerate}
        \item[\rm{(1)}] $\sigma_a(b)\circ \rho_b(a)=a \circ b$, for all $a, b \in S$;
        \item[\rm{(2)}] the map $\rho: S \to S^S, b \mapsto \rho_b$
    is an anti-homomorphism from the inverse semigroup $(S, \circ)$ into the monoid $S^S$.
    \end{enumerate}
\end{prop}
\begin{proof}
    First, for all $a, b \in S$, 
we have
\begin{align*}
    \sigma_a(b)\circ \rho_b(a)&
    =\sigma_a\left(b \circ b^-\right)\circ a \circ b=a \circ b \circ b^- \circ a^- \circ a \circ b=a\circ b,
\end{align*}
proving point (1). This implies that, for all $a, b \in S$,
\begin{align}\label{rho_rhomeno}
    \rho_b(a)^- \circ \rho_b(a)=b^- \circ a^- \circ a \circ b.
\end{align}
In fact, if $a, b \in S$,
\begin{align*}
    \rho_b(a)^- \circ \rho_b(a)&=b^- \circ a^- \circ \sigma_a(b)\circ \rho_b(a)=b^- \circ a^- \circ a \circ b.
\end{align*}
Thus, if $a, b, c \in S$, we get
\begin{align*}
    \left(\rho_{b \circ c}(a)\right)^-&=c^- \circ b^- \circ a^- \circ a \circ \left(a^-+b \circ c\right)\\
    &=c^- \circ b^- \circ a^- - c^- \circ b^- \circ a^-\circ a+ c^- \circ b^- \circ a^- \circ a\circ b \circ c
\end{align*}
and, using \eqref{rho_rhomeno},
\begin{align*}
   &\left( \rho_c\rho_b(a)\right)^-\\
   &=c^- \circ \rho_b(a)^--c^- \circ \rho_b(a)^- \circ \rho_b(a)+c^- \circ \rho_b(a)^- \circ \rho_b(a) \circ c\\
   &= c^- \circ \rho_b(a)^-- c^- \circ b^- \circ a^- \circ a \circ b+ c^- \circ b^- \circ a^- \circ a \circ b \circ c\\
   &=c^- \circ b^- \circ a^- \circ a \circ \left(a^-+b\right)- c^- \circ b^- \circ a^- \circ a \circ b+ c^- \circ b^- \circ a^- \circ a \circ b \circ c\\
   &=c^- \circ b^- \circ a^--c^- \circ b^- \circ a^- \circ a+c^- \circ b^- \circ a^- \circ a\circ b\\
   &\quad - c^- \circ b^- \circ a^- \circ a \circ b+ c^- \circ b^- \circ a^- \circ a \circ b \circ c\\
   &=c^- \circ b^- \circ a^- - c^- \circ b^- \circ a^-\circ a+ c^- \circ b^- \circ a^- \circ a\circ b \circ c 
\end{align*}
where the last equality follows from \cref{lemma_ab}(2).  
Therefore, the claim in (2) follows.
\end{proof}

\smallskip

\begin{theor}\label{teo_solu}
  Let $(S, +, \circ)$ be an inverse brace such that the map $\sigma: S \to S^S, a \mapsto \sigma_a$ is a homomorphism from the inverse semigroup $(S, \circ)$ to the monoid $S^S$, namely \eqref{A} is satisfied. Then the map $r_S:S \times S\to S\times S$ defined by $$r_S(a,b)=\left(\sigma_a(b), \sigma_a(b)^- \circ a \circ b\right),$$
 for all $a, b \in S$, is a solution if and only if the condition
    \begin{align}
    &\sigma_a(b)\circ \sigma_a(b)^-=\sigma_a\left(b \circ b^-\right), 
    \tag{II}
    \end{align}
    is satisfied, for all $a, b \in S$.
\end{theor}
\begin{proof}
First, assume that \eqref{B} holds. It is a routine computation to check that the map $r_S(a,b)=(\sigma_a(b), \rho_b(a))$ is a solution if and only if the following three fundamental equalities hold
\begin{align}
     &\label{first} \sigma_a\sigma_b(c)=\sigma_{\sigma_a\left(b\right)}\sigma_{\rho_b\left(a\right)}\left(c\right)\tag{Y1},\\
    &  \label{second}\sigma_{\rho_{\sigma_b\left(c\right)}\left(a\right)}\rho_c\left(b\right)=\rho_{\sigma_{\rho_b\left(a\right)}\left(c\right)}\sigma_a\left(b\right)\tag{Y2},\\
      &\label{third}\rho_c\rho_b(a)=\rho_{\rho_c\left(b\right)}\rho_{\sigma_b\left(c\right)}\left(a\right)\tag{Y3},
  \end{align} 
for all $a,b,c \in S$. \\
For $a, b, c \in S$, using \cref{condizione_a_2}, we obtain 
\begin{align*}
    \sigma_a\sigma_b(c)= \sigma_{a \circ b}(c)=\sigma_{\sigma_a(b)\circ \rho_b(a)}(c)=\sigma_{\sigma_a\left(b\right)}\sigma_{\rho_b\left(a\right)}\left(c\right)
\end{align*}
and
\begin{align*}
    \rho_c\rho_b(a)= \rho_{b \circ c}(a)=\rho_{\sigma_b(c)\circ \rho_c(b)}(a)=\rho_{\rho_c\left(b\right)}\rho_{\sigma_b\left(c\right)}\left(a\right).
\end{align*}
Hence, both \eqref{first} and \eqref{third} hold.
Moreover, 
\begin{align*}
    \rho_{\sigma_{\rho_b(a)}(c)}\sigma_a(b)&=\left(\sigma_{\sigma_a(b)}\sigma_{\rho_b(a)}(c)   \right)^-\circ \sigma_a(b)\circ\sigma_{\rho_b(a)}(c)
    \\&=\left(\sigma_{a}\sigma_b(c)\right)^- \circ \sigma_a(b)\circ \sigma_{\rho_b(a)}(c) &\mbox{by \eqref{first}}\\
    &=\left(\sigma_{a}\sigma_b(c)\right)^- \circ \sigma_a(b)\circ \rho_b(a)\circ c \circ\left(\rho_c\rho_b(a)\right)^-  &\mbox{by \cref{sigma_rho}}\\
    &=\left(\sigma_{a}\sigma_b(c)\right)^- \circ a \circ b \circ c \circ \left(\rho_c\rho_b(a)\right)^-  &\mbox{by \cref{condizione_a_2}(1)}\\
    &=\left(\sigma_{a}\sigma_b(c)\right)^- \circ a \circ \sigma_b(c) \circ \rho_c(b) \circ \left(\rho_c \rho_b(a)   \right)^-&\mbox{by \cref{condizione_a_2}(1)}\\
    &=\rho_{\sigma_b(c)}(a) \circ \rho_c(b) \circ \left(\rho_{\rho_c(b)}\rho_{\sigma_b(c)}(a)\right)^-&\mbox{by \eqref{third}}\\
&=\sigma_{\rho_{\sigma_b\left(c\right)}\left(a\right)}\rho_c\left(b\right). &\mbox{by \cref{sigma_rho}}
\end{align*}
Therefore, $r_S$ is a solution.\\
Conversely, assume that $r_S$ is a solution. From equation \eqref{first}, taking $c=-b^-+b^-$ we obtain
   \begin{align*}
\sigma_a\sigma_b\left(-b^-+b^-\right)=\sigma_{\sigma_a\left(b\right)}\sigma_{\rho_b\left(a\right)}\left(-b^-+b^-\right).  
   \end{align*}
   The left-hand side is equal to $
  \sigma_a\sigma_b\left(-b^-+b^-\right)=\sigma_a\left(b \circ \left( b^--b^-+b^-\right)\right)=\sigma_a\left(b \circ b^- \right)$. About the right-hand side, we have
  \begin{align*}
&\sigma_{\sigma_a\left(b\right)}\sigma_{\rho_b\left(a\right)}\left(-b^-+b^-\right)\\
   &=\sigma_{\sigma_a\left(b\right)}\left(\rho_b(a) \circ \left(b^- \circ a^- \circ \sigma_a(b)-b^-+b^-\right)\right)\\
   &=\sigma_{\sigma_a\left(b\right)}\left(\rho_b(a) \circ b^- \circ a^- \circ \sigma_a(b)\right) &\mbox{by \cref{lemma_1}(1)}\\
&=\sigma_{\sigma_a\left(b\right)}\left(\rho_b(a) \circ \rho_b(a)^-\right)\\
  &=\sigma_a(b) \circ \sigma_a(b)^--\sigma_a(b)+ \sigma_a(b) \circ \rho_b(a) \circ \rho_b(a)^-\\
    &=\sigma_a(b) \circ \sigma_a(b)^--\sigma_a(b)+ a \circ b \circ b^- \circ a^- \circ \sigma_a(b)\\
   &=\sigma_a(b) \circ \sigma_a(b)^--\sigma_a(b)+a \circ b \circ b^- \circ \left(a^-+b\right)\\
   &= \sigma_a(b) \circ \sigma_a(b)^--\sigma_a(b)+\sigma_a(b) &\mbox{by \cref{A*}}\\
   &=\sigma_a(b) \circ \sigma_a(b)^- &\mbox{by \cref{lemma_1}(1)}
  \end{align*}
Therefore, \eqref{B} is satisfied.
\end{proof}

\smallskip

\begin{rem}
 Observe that whenever $(S,+, \circ)$ is a weak brace, condition \eqref{B} is satisfied, since $r_S$ always defines a solution.
\end{rem}

\smallskip

\begin{cor}\label{cor_quasi_non_deg}
    Let $(S, +, \circ)$ be an inverse brace satisfying \eqref{A} and \eqref{B} and such that $(S, \circ)$ is a Clifford semigroup. Then the map $r_S:S \times S\to S\times S$ defined by $r_S(a,b)=\left(\sigma_a(b), \rho_b(a)\right)$,
 for all $a, b \in S$, is quasi non-degenerate solution.
    \begin{proof}
       Clearly, by \cref{teo_solu}, the map $r_S$ is a solution. Moreover, $r_S$ is quasi non-degenerate, since the inverses of the maps $\sigma_a$ and $\rho_a$ are precisely the maps $\sigma_{a^-}$ and $\rho_{a^-}$, respectively, as a consequence of
       \cref{Aequivsigma} and \cref{condizione_a_2}(2).
    \end{proof}
\end{cor}

\smallskip

\begin{ex}\label{ex-4ele}
    Let $S=\{a,b,c,d\}$ and consider the following operations on $S$:
    $$\begin{tabular}{c | c c c c c  c}
    $+$ & $a$ & $b$ & $c$  &$d$ \\
    \hline
   $a$& $a$ & $a$ & $a$ &$d$\\
    $b$ & $a$ & $b$ & $c$ & $d$\\
    $c$& $a$ & $c$ & $b$ & $d$\\
    $d$ & $d$& $d$& $d$& $a$
\end{tabular}  \qquad \qquad{\rm and}\qquad \qquad \begin{tabular}{c | c c c c c  c}
    $\circ$ & $a$ & $b$ & $c$ &$d$  \\
    \hline
   $a$& $a$ & $a$ & $a$ & $d$\\
    $b$ & $a$ & $b$ & $c$ & $d$\\
    $c$& $a$ & $c$ & $c$ & $d$\\
    $d$& $d$ & $d$ & $d$ & $a$
\end{tabular}$$
Then $(S, +, \circ)$ is an inverse brace 
 satisfying \eqref{A} and \eqref{B}. Thus, by \cref{teo_solu}, the map $r_S$ is a solution. Moreover, since $(S, \circ)$ is Clifford, by \cref{cor_quasi_non_deg}, the map $r_S$ is quasi non-degenerate.

 A short \texttt{GAP} routine verifies that there are $11$ inverse braces of size $4$ satisfying the identities \eqref{A} and \eqref{B} that are not weak braces. 
\end{ex}

\smallskip

In the particular case of inverse braces $(S, +, \circ)$ whose multiplicative structure is a semilattice, the condition \eqref{A} is trivially satisfied (see \cref{ex_semilattice}) and identity \eqref{B} follows immediately by the fact that $x^-=x$, for all $x \in S$. Thus, this family of inverse braces always give rise to quasi non-degenerate solutions of the form in \eqref{r_S}.
\begin{cor}\label{cor_semilattice}
Let $(S, +, \circ)$ be an inverse brace such that $(S, \circ)$ is a semilattice. Then the map $r_S:S \times S\to S\times S$ given by $r_S(a,b)=\left(\sigma_a(b), \rho_b(a)\right)$,
 for all $a, b \in S$, is a quasi non-degenerate solution. 
\end{cor}

\smallskip
The following are applications of \cref{cor_semilattice}.
 \begin{exs}\hspace{1mm}
 \begin{enumerate}
 \item The map $r_{\mathbb{Z}_2}$ associated with the inverse brace $(\mathbb{Z}_2, +, \circ)$ in \cref{esempiowang} is a quasi non-degenerate solution. 
     \item The map $r_S$ associated with the inverse brace $(S, +, \circ)$ in  \cref{ex_cliff} is a quasi non-degenerate solution. 
 \end{enumerate}
 \end{exs}

\smallskip

The following result is consistent with \cite[Proposition 14]{CaMaMiSt22} and provides information on the order of the solution $r_S$ in certain special cases.
\begin{prop} \label{commutativo} 
    Let $(S,+,\circ)$ be an inverse brace  satisfying \eqref{A} and \eqref{B} and such that $(S,+)$ is
    a commutative semigroup. Then, the map $r_S:S \times S\to S\times S$ given by $r_S(a,b)=\left(\sigma_a(b), \rho_b(a)\right)$,
 for all $a, b \in S$, is a cubic solution, that is, $r_S^3=r_S.$\\
     In particular, $r_S$ is bijective if and only if $S$ is a brace. 
\end{prop}
\begin{proof}
Initially, by \cref{teo_solu}, the map $r_S$ is a solution.  Now, let $a,b \in S$ and 
    let us analyze separately the two components of $$r_S^3(a,b)=\left(\sigma_{\sigma_{\sigma_a(b)}\rho_b(a)}\rho_{\rho_b(a)}\sigma_a(b), \, \rho_{\rho_{\rho_b(a)}\sigma_a(b)}\sigma_{\sigma_a(b)}\rho_b(a)\right).$$
    Focusing on the first one, we have
    \begin{align*}
    &\sigma_{\sigma_a(b)}\rho_b(a)\circ\left( \sigma_{\sigma_a(b)}\rho_b(a)^-+ \rho_{\rho_b(a)}\sigma_a(b)\right)\\
    &=\sigma_{\sigma_a(b)}\rho_b(a)\circ \sigma_{\sigma_a(b)}\rho_b(a)^--\sigma_{\sigma_a(b)}\rho_b(a) + a \circ b &\mbox{by \cref{condizione_a_2}(1)}\\
    & =\sigma_a(b)\circ\rho_b(a)\circ\rho_b(a)^-\circ\sigma_a(b)^- -\sigma_{\sigma_a(b)}\rho_b(a) + a \circ b &\mbox{by \eqref{A}-\eqref{B}}\\
    &=a \circ b \circ b^- \circ a^- -\sigma_{\sigma_a(b)}\rho_b(a) + a \circ b &\mbox{by \cref{condizione_a_2}(1)}\\
        &={a \circ b \circ b^- \circ a^- -\sigma_a(b) \circ \sigma_a(b)^- + \sigma_a(b) -\sigma_a(b) \circ \rho_b(a) + a \circ b}\\
        &={a \circ b \circ b^- \circ a^- -a \circ b + a \circ b -\sigma_a(b) \circ \sigma_a(b)^- + \sigma_a(b)} &\mbox{by \cref{condizione_a_2}(1)}\\
         &={a \circ b \circ b^- \circ a^- -\sigma_a(b) \circ \sigma_a(b)^- + \sigma_a(b)} &\mbox{{by \cref{lemma_1}(1)}}\\
         &={\sigma_a(b) \circ \sigma_a(b)^--\sigma_a(b) \circ \sigma_a(b)^-+ \sigma_a(b)} &\mbox{by \eqref{A}-\eqref{B}}\\
         &=\sigma_a(b) &\mbox{by \cref{lemma_ab}(2)}
    \end{align*}
Using the computation of the first component, the second component is given by:
        \begin{align*}
         &\rho_{\rho_{\rho_b(a)}\sigma_a(b)}\sigma_{\sigma_a(b)}\rho_b(a)\\
        &=\left(\sigma_{\sigma_{\sigma_a(b)}\rho_b(a)}\rho_{\rho_b(a)}\sigma_a(b)\right)^- \circ \sigma_{\sigma_a(b)}\rho_b(a) \circ \rho_{\rho_b(a)}\sigma_a(b)
        \\
        &=\sigma_a(b)^- \circ a \circ b &\mbox{by \cref{condizione_a_2}(1)}\\
    &=\rho_b(a).
    \end{align*} 
    Therefore, $r_S^3=r_S.$

    Now, suppose that $r_S$ is bijective. Since $r_S$ is cubic, it is involutive. Thus, this implies that $a=\sigma_{\sigma_a(b)}\rho_b(a)$ and $b=\rho_{\rho_b(a)}\sigma_a(b)$, for all $a, b \in S$.
    By \cref{condizione_a_2}(1), it follows that, for all $a,b \in S,$
    \begin{align*}
        b&=\rho_{\rho_b(a)}\sigma_a(b)=\left(\sigma_{\sigma_a(b)}\rho_b(a)\right)^-\circ \sigma_a(b)\circ \rho_b(a) 
        =a^-\circ a\circ b
    \end{align*}
and
    \begin{align*}
        a&=\sigma_a(b)\circ\rho_b(a)\circ \left(\rho_{\rho_b(a)}\sigma_a(b)\right)^-=a\circ b \circ b^-.
    \end{align*}
   Thus, for all $a,b \in S$, we get
    \[b\circ b^-=a^-\circ a\circ b \circ b^-=a^-\circ a,\]
     hence $(S,\circ)$ is a group.
    Therefore, by \cref{circ_gruppo}, $(S,+,\circ)$ is a brace.\\
    The converse follows by \cite{Ru07} and \cite[Lemma 2]{CeJeOk14} since the solution associated with any brace is involutive.
\end{proof}

\smallskip

\begin{rem}
    As a consequence of \cref{commutativo}, if $(S,+,\circ)$ is an inverse brace satisfying \eqref{A} and \eqref{B} and such that $(S,+)$ is
    a commutative semigroup, then the solution $r_S$ is quasi bijective with $r_S^-=r_S$.
\end{rem}

\medskip

\section{Constructions of inverse braces}
This section presents methods for constructing inverse braces from given ones, following the approach used for weak braces in \cite[Section 4]{CaMaMiSt22} and similar algebraic structures in \cite{CaMaSt21}. We show that these constructions preserve the existence of solutions; that is, if the original inverse brace admits a solution, then so does the constructed one.

\medskip

\subsection{Matched product of inverse braces}

Let us begin by introducing the following definition, which is consistent with \cite[Definition 17]{CaMaMiSt22}.
\begin{defin}\label{def_matched}
Let $S$ and $T$ be inverse braces, $\alpha:T\rightarrow \Aut(S)$ a homomorphism of inverse semigroups from $(T,\circ)$ into $\Aut(S,+)$, 
and $\beta:S\rightarrow \Aut(T)$ a homomorphism of inverse semigroups from $(S,\circ)$ into $\Aut(T,+)$,
such that for all $a,b \in S$ and $u,v \in T$ the following conditions hold:
\begin{align*}
    &\alpha_u\left(\alpha_u^{-1} (a)\circ b\right)= a\circ \alpha_{\beta_{a}^{-1}(u)}(b),\\
 &\beta_a\left(\beta_a^{-1}(u)\circ v\right)= u\circ \beta_{\alpha_{u}^{-1}(a)} (v),\\
&\alpha_u\left(\alpha_u^{-1} (a)\circ a\right)= a\ \Rightarrow \ \alpha_u(a)=a,\\
&  \beta_a\left(\beta_a^{-1}(u)\circ u\right)=u\ \Rightarrow \  \beta_a(u)=u.
\end{align*}
Then $(S, T,\alpha,\beta)$ is called \emph{matched product system of inverse braces}.
\end{defin}

\smallskip

\begin{rem}\label{rem:idemp-alpha-beta}
   Note that if $e\in \E(T, \circ)$, then $\alpha_{e} = \id_S$.
    In fact, if $u\in T$ is such that $e = u\circ u^{-}$, then
    \begin{align*}
        \alpha_{e}
        = \alpha_{u\circ u^{-}}
        = \alpha_{u}\alpha_{u^{-}}
        = \alpha_{u}\alpha^{-1}_{u}=\id_S.
    \end{align*}
Similarly, if $f\in \E(S, \circ )$, then $\beta_{f} = \id_T$.
   \end{rem}

   \smallskip

\begin{theor}\label{18}
Let $(S, T, \alpha,\beta)$ be a matched product system of inverse braces.  
Define the following operations
\begin{align*}
(a,u) + (b,v) &:= (a + b, u + v),\\
(a,u) \circ (b,v) &:= \left(\alpha_u\left(\alpha_u^{-1}(a)\circ b\right),\ \beta_a\left(\beta_a^{-1}(u)\circ v\right)\right),
\end{align*}
for all $(a,u),(b,v) \in S \times T$.
    Then $(S \times T, +, \circ)$ is an inverse brace, called the \emph{matched product of $S$ and $T$ via $\alpha$ and $\beta$}, and denoted by $S \bowtie T$. 
\end{theor}
\begin{proof}
Clearly, $(S \times T,+)$ is an inverse semigroup, as it is the direct product of the inverse semigroups $(S,+)$ and $(T,+)$. Moreover, proceeding as in the proof of \cite[Theorem 12]{CaMaSt21}, $(S \times T,\circ)$ is an inverse semigroup with $(a, u)^-=\left(\alpha^{-1}_{\beta^{-1}_a(u)}(a^-),\beta^{-1}_{\alpha^{-1}_u(a)}(u^-)\right)$, for all $(a, u) \in S \times T$. Further, if $(a,u),(b,v),(c,w) \in S \times T$, then
\begin{align*}
&(a,u) \circ \left((b,v)+(c,w)\right)\\
&=(a,u)\circ(b+c,\, v+w)\\
&=\left(\alpha_u\left(\alpha^{-1}_u(a)\circ (b+c)\right),\,
       \beta_a\left(\beta^{-1}_a(u)\circ (v+w)\right)\right)\\
&=\left(\alpha_u\left(\alpha^{-1}_u(a)\circ b-\alpha^{-1}_u(a)+\alpha^{-1}_u(a)\circ c\right), \ \beta_a\left(\beta^{-1}_a(u)\circ v-\beta^{-1}_a(u)+\beta^{-1}_a(u)\circ w\right)\right)\\
&=\left(\alpha_u\left(\alpha^{-1}_u(a)\circ b\right)-a+\alpha_u\left(\alpha^{-1}_u(a)\circ c\right),\, \beta_a\left(\beta^{-1}_a(u)\circ v\right)-u+\beta_a\left(\beta^{-1}_a(u)\circ w\right)\right)\\
&=\left(\alpha_u\left(\alpha^{-1}_u(a)\circ b\right),\,\beta_a\left(\beta^{-1}_a(u)\circ v\right)\right)
 \, -(a,u) \,
 +\left(\alpha_u\left(\alpha^{-1}_u(a)\circ c\right),\,\beta_a\left(\beta^{-1}_a(u)\circ w\right)\right)\\
&=(a,u) \circ (b,v) -(a,u) + (a,u) \circ (c,w).
\end{align*}
Therefore, the claim follows. 
\end{proof}

\smallskip



Now, we aim to show that the map $r_{S\bowtie T}$ associated to a matched product $S \bowtie T$ of two inverse braces satisfying both \eqref{A} and \eqref{B} is a solution. Moreover, we prove that  $r_{S\bowtie T}$ coincides with the matched product of the solutions associated with the inverse braces $S$ and $T$, respectively, namely, $r_{S\bowtie T}=r_S \bowtie r_T$.\\
For the ease of the reader, let us first recall the notion introduced in \cite{CCoSt20}. 

\begin{defin} Let $S$ and $T$ be two sets, $r_S(a, b)=(\lambda_a(b), \rho_b(a))$ a solution on $S$, and $r_T(u, v)=(\lambda_u(v), \rho_v(u))$ a solution on $T$. Let $\alpha: T \to \Sym\left(S\right)$ and $\beta: S \to \Sym\left(T\right)$ be two maps, and set $\alpha_u:=\alpha\left(u\right)$, for every $u \in T$, and $\beta_a:=\beta\left(a\right)$, for every $a\in S$. Then the quadruple $\left(r_S,r_T, \alpha,\beta\right)$ is said to be a \emph{matched product system of solutions} if the following conditions hold
	{
		\begin{center}
			\begin{minipage}[b]{.5\textwidth}
				\vspace{-\baselineskip}
				\begin{align}\label{eq:primo}\tag{s1}
					\alpha_u\alpha_v = \alpha_{\lambda_u\left(v\right)}\alpha_{\rho_{v}\left(u\right)}
				\end{align}
			\end{minipage}%
			\hfill\hfill\hfill
			\begin{minipage}[b]{.5\textwidth}
				\vspace{-\baselineskip}			\begin{align}\label{eq:secondo}\tag{s2}			\beta_a\beta_b=\beta_{\lambda_a\left(b\right)}\beta_{\rho_b\left(a\right)}
				\end{align}
			\end{minipage}
		\end{center}
		\begin{center}
			\begin{minipage}[b]{.5\textwidth}
				\vspace{-\baselineskip}
				\begin{align}\label{eq:quinto}\tag{s3}
					\rho_{\alpha^{-1}_u\!\left(b\right)}\alpha^{-1}_{\beta_a\left(u\right)}\left(a\right) = \alpha^{-1}_{\beta_{\rho_b\left(a\right)}\beta^{-1}_b\left(u\right)}\rho_b\left(a\right)
				\end{align}
			\end{minipage}%
			\hfill\hfill
			\begin{minipage}[b]{.5\textwidth}
				\vspace{-\baselineskip}
				\begin{align}\label{eq:sesto}\tag{s4}
					\rho_{\beta^{-1}_a\!\left(v\right)}\beta^{-1}_{\alpha_u\left(a\right)}\left(u\right) = \beta^{-1}_{\alpha_{\rho_v\left(u\right)}\alpha^{-1}_v\left(a\right)}\rho_v\left(u\right)
				\end{align}
			\end{minipage}
		\end{center}
		\begin{center}
			\begin{minipage}[b]{.5\textwidth}
				\vspace{-\baselineskip}
				\begin{align}\label{eq:terzo}\tag{s5}
					\lambda_a\alphaa{}{\beta^{-1}_{a}\left(u\right)}{}= \alphaa{}{u}{\lambdaa{\alphaa{-1}{u}{\left(a\right)}}{}}
				\end{align}
			\end{minipage}%
			\hfill\hfill
			\begin{minipage}[b]{.5\textwidth}
				\vspace{-\baselineskip}
				\begin{align}\label{eq:quarto}\tag{s6}
					\lambdaa{u}{\betaa{}{\alphaa{-1}{u}{\left(a\right)}}{}}=\betaa{}{a}{\lambdaa{\beta^{-1}_{a}\left(u\right)}{}}
				\end{align}
			\end{minipage}
		\end{center}
	}
	\noindent for all $a,b \in S$ and $u,v \in T$.
    \end{defin}

As shown in \cite[Theorem 2]{CCoSt20}, any matched product system of solutions determines a new solution on the set $S\times T$, defined by
		\begin{align}\label{matchsolu}
			& r_S\bowtie r_T\left(\left(a, u\right), 
			\left(b, v\right)\right) := 
			\left(\left(\alphaa{}{u}{\lambdaa{\bar{a}}{\left(b\right)}},\, \beta_a\lambdaa{\bar{u}}{\left(v\right)}\right),\ \left(\alphaa{-1}{\overline{U}}{\rhoo{\alphaa{}{\bar{u}}{\left(b\right)}}{\left(a\right)}},\,  \beta^{-1}_{\overline{A}}\rhoo{\beta_{\bar{a}}\left(v\right)}{\left(u\right)}\right) \right),
		\end{align}	
		\noindent where we set
		\begin{center}
		   $\bar{a}:=\alphaa{-1}{u}{\left(a\right)}$, \,\,$\bar{u}:= \beta^{-1}_{a}\left(u\right)$,\,\, $A:=\alphaa{}{u}{\lambdaa{\bar{a}}{\left(b\right)}}$,\, $U:=\beta_a\lambdaa{\bar{u}}{\left(v\right)}$,\,\, $\overline{A}:=\alphaa{-1}{U}{\left(A\right)}$,\,\, $\overline{U}:= \beta^{-1}_{A}\left(U\right)$,
		\end{center}
		for all $\left(a,u\right),\left(b,v\right)\in S\times T$. 
		This solution is called the \emph{matched product of the solutions} $r_S$ and $r_T$ (via $\alpha$ and $\beta$).

\smallskip

\begin{theor}\label{teo_match}
   Let $(S, T,\alpha,\beta)$ be a matched product system of inverse braces that satisfy both \eqref{A} and \eqref{B}. Also, assume that 
   \begin{align}\label{hp_match}
    \alpha_{\beta_a^{-1}(u)}\left(b \circ b^-\right)&=\alpha_{\beta_a^{-1}(u)}(b) \circ \alpha_{\beta_a^{-1}(u)}(b)^-\\
    \label{hp_match1} \beta_{\alpha_u^{-1}(a)}\left(v \circ v^-\right)&=\beta_{\alpha_u^{-1}(a)}(v) \circ \beta_{\alpha_u^{-1}(a)}(v)^-
   \end{align}
   hold, for all $(a, u), (b, v) \in S \times T$.
   Then the map $r_{S \ \bowtie \ T}: (S \bowtie T) \times (S \bowtie T) \to (S \bowtie T) \times (S \bowtie T)$ defined by
   \begin{align*}
 r_{S \ \bowtie \ T}\left(\left(a,u\right), \left(b, v\right) \right)=\left(\sigma_{(a,u)}(b, v), \sigma_{(a,u)}(b, v)^- \circ (a, u) \circ (b, v)\right),
   \end{align*}
   for all $(a, u), (b, v) \in S \times T$, is a solution.
   \begin{proof}
Initially, to prove that $r_{S \bowtie T}$ is a solution, it suffices, by \cref{teo_solu}, to show that \eqref{A} and \eqref{B} are also satisfied by the inverse brace $S \bowtie T$. \\
For $a,b \in S$ and for all $u,v \in T$, we have 
       \begin{align*}
       &\sigma_{(a,u)}\left(\left(b,v\right)\circ\left(b,v\right)^-\right)
       \\
      &= \left(a\circ a^--a+ a\circ \alpha_{\beta_a^{-1}(u)}\left(b \circ b^-\right), \, u\circ u^--u+u\circ \beta_{\alpha_u^{-1}(a)}\left(v \circ v^- \right)\right)\\
       &=\left(\sigma_a \left(\alpha_{\beta_a^{-1}(u)}\left(b \circ b^-\right)\right), \ \sigma_u \left(\beta_{\alpha_u^{-1}(a)}\left(v \circ v^- \right)\right)\right)\\
       &=\left(a\circ \alpha_{\beta_a^{-1}(u)}(b) \circ \alpha_{\beta_a^{-1}(u)}(b)^-\circ a^-, \, u\circ\beta_{\alpha_u^{-1}(a)}(v) \circ \circ\beta_{\alpha_u^{-1}(a)}(v)^- \circ u^-\right)&\mbox{by \eqref{hp_match}-\eqref{hp_match1}-\eqref{A}}\\
       &=(a,u)\circ(b,v) \circ \left((a,u)\circ(b,v)\right)^-.
              \end{align*}      
   Hence, \eqref{A} is proved.           
Moreover, for $a,b \in S$ and for $u,v \in T$,
\begin{align*}
   & \sigma_{(a,u)}(b, v) \circ \sigma_{(a,u)}(b, v)^-\\
   &=\left( \sigma_a\alpha_{\beta^{-1}_a(u)}(b) \circ \sigma_a\alpha_{\beta^{-1}_a(u)}(b)^-, \, \sigma_u\beta_{\alpha^{-1}_u(a)}(v) \circ \sigma_u\beta_{\alpha^{-1}_u(a)}(v)^-\right)    \\
  &=\left( \sigma_a\left(\alpha_{\beta^{-1}_a(u)}(b)  \circ \alpha_{\beta^{-1}_a(u)}(b) ^-\right) , \, \sigma_u\left(\beta_{\alpha^{-1}_u(a)}(v) \circ \beta_{\alpha^{-1}_u(a)}(v)^-\right)\right) &\mbox{by \eqref{B}}\\
    &=\left( \sigma_a\left(\alpha_{\beta^{-1}_a(u)}(b \circ b^-) \right) , \, \sigma_u\left(\beta_{\alpha^{-1}_u(a)}(v \circ v^-)\right)\right) &\mbox{by \eqref{hp_match}-\eqref{hp_match1}}
\end{align*}
and, thus, proceeding as above, we get 
\begin{align*}
   \sigma_{(a,u)}(b, v) \circ \sigma_{(a,u)}(b, v)^-=(a,u)\circ (b,v)\circ \left((a,u)\circ (b,v)\right)^-,
\end{align*}
namely, also \eqref{B} is satisfied. Therefore, $r_{S \ \bowtie \ T}$ is a solution.
   \end{proof}
\end{theor}

\smallskip

\begin{rem}
We observe that if $S$ and $T$ are weak braces, then the conditions \eqref{hp_match} and \eqref{hp_match1} in \cref{teo_match} are automatically satisfied. In fact, for example,
    \begin{align*}
        \alpha_{\beta_a^{-1}(u)}\left(b \circ b^-\right)&=\alpha_{\beta_a^{-1}(u)}\left(-b+b\right)=-\alpha_{\beta_a^{-1}(u)}\left(b\right)+\alpha_{\beta_a^{-1}(u)}\left(b\right)\\
        &=\alpha_{\beta_a^{-1}(u)}(b) \circ \alpha_{\beta_a^{-1}(u)}(b)^-,
    \end{align*}
    for all $a, b \in S$ and $u \in T$. The verification of condition \eqref{hp_match1} is analogous.
\end{rem}

\smallskip

\begin{theor}
   Let $(S, T,\alpha,\beta)$ be a matched product system of inverse braces that satisfy both \eqref{A} and \eqref{B}. Also, assume that \eqref{hp_match} and \eqref{hp_match1}
   hold.
   Then  $r_{S \ \bowtie \ T}=r_S \bowtie r_T$.
   \end{theor}
   \begin{proof}
Initially,  we show that $\left(r_S,r_T, \alpha,\beta\right)$ is a matched product system of solutions. It is enough to verify only identities \eqref{eq:primo}, \eqref{eq:quinto}, and \eqref{eq:terzo}, since the other ones can be obtained similarly, reversing the role of $\alpha$ and $\beta$, $a$ and $u$, and $b$ and $v$. For all $(a, u), (b, v) \in S   \times T$, by \cref{condizione_a_2}(1), we have
\begin{align*}
    \alpha_u\alpha_v = \alpha_{u \circ v}=\alpha_{\sigma_u\left(v\right) \circ \rho_{v}\left(u\right)}=\alpha_{\sigma_u\left(v\right)}\alpha_{\rho_{v}\left(u\right)},
\end{align*}
namely, \eqref{eq:primo} is satisfied. 
Moreover, for all $(a, u), (b, v) \in S \times T$,  we have
\begin{align*}
    &\sigma_{(a, u)}(b, v)=\left(\sigma_a\alpha_{\beta^{-1}_a(u)}(b), \, \sigma_u\beta_{\alpha^{-1}_u(a)}(v)\right).
\end{align*}
According to \cite[Lemma 11]{CaMaSt21}, observing that, for all $(a, u) \in S  \times T$,
      \begin{align}\label{inv}
 \alpha^{-1}_{u}\left(a\right)^{-}
    = \alpha^{-1}_{\beta^{-1}_{a}\left(u\right)}\left(a^{-}\right)\qquad \text{and} \qquad
    \beta^{-1}_{a}\left(u\right)^{-}
    = \beta^{-1}_{\alpha^{-1}_{u}\left(a\right)}\left(u^{-}\right),
\end{align}
we compute
\begin{align*}
    \sigma_a\alpha_{\beta^{-1}_a(u)}(b)&=a \circ \left(a^-+\alpha_{\beta^{-1}_a(u)}(b)\right)\\
    &=a \circ \left(\alpha_{\beta^{-1}_a(u)}\left(\alpha_u^{-1}(a)^-+b\right)\right) &\mbox{by \eqref{inv}}\\
    &=\alpha_u\left(\alpha_u^{-1}(a) \circ  \left(\alpha_u^{-1}(a)^-+b\right)\right)\\
    &=\alpha_u\sigma_{\alpha_u^{-1}(a)}(b)
\end{align*}
and, similarly, $\sigma_u\beta_{\alpha^{-1}_u(a)}(v)=\beta_a\sigma_{\beta_{a}^{-1}(u)}(v)$. Thus, \eqref{eq:terzo} is valid. 
Moreover,
\begin{align*}
    \rho_{\alpha^{-1}_u\!\left(b\right)}&\alpha^{-1}_{\beta_a\left(u\right)}\left(a\right) \\&=\rho_{\alpha^{-1}_u\!\left(b\right)}\left(\alpha_u^{-1}\left(a^-\right)^-\right)&\mbox{by \eqref{inv}}\\
    &=\left(\alpha_u^{-1}\left(a^-\right)+ \alpha^{-1}_u\!\left(b\right)\right)^- \circ \alpha_u^{-1}\left(a^-\right) \circ \alpha_u^{-1}\left(a^-\right)^- \circ \alpha^{-1}_u\!\left(b\right)\\
    &=\left(\alpha_u^{-1}\left(a^-+b\right)\right)^- \circ \alpha_u^{-1}\left(a^-\right) \circ \alpha_u^{-1}\left(a^-\right)^- \circ \alpha^{-1}_u\!\left(b\right)\\
    &=\left(\alpha_u^{-1}\left(a^-+b\right)\right)^- \circ \alpha_u^{-1}\left(a^- \circ a \right) \circ \alpha^{-1}_u\!\left(b\right) &\mbox{by \eqref{hp_match}}\\
    &=\left(\alpha^{-1}_u\!\left(b\right)^- \circ \alpha_u^{-1}\left(a^- \circ a \right)^- \circ  \alpha_u^{-1}\left(a^-+b\right)\right)^-\\
    &=\left(\alpha^{-1}_{\beta^{-1}_b\left(u\right)}\left(b^-\right) \circ \alpha_u^{-1}\left(a^- \circ a \right)^- \circ  \alpha_u^{-1}\left(a^-+b\right)\right)^-&\mbox{by \eqref{inv}}\\
     &=\left(\alpha^{-1}_{\beta^{-1}_b\left(u\right)}\left(b^- \circ \alpha_{\beta^{-1}_{b^-}\beta^{-1}_b(u)}\left(\alpha_u^{-1}\left(a^- \circ a \right)^- \circ  \alpha_u^{-1}\left(a^-+b\right)\right)\right)\right)^- &\mbox{by \cref{def_matched}}\\
     &=\left(\alpha^{-1}_{\beta^{-1}_b\left(u\right)}\left(b^- \circ \alpha_{u}\left(\alpha_u^{-1}\left(a^- \circ a \right)^- \circ  \alpha_u^{-1}\left(a^-+b\right)\right)\right)\right)^- \\
     &=\left(\alpha^{-1}_{\beta^{-1}_b\left(u\right)}\left(b^- \circ a  \circ a^- \circ \alpha_{\beta^{-1}_{a  \circ a^-}\left(u\right)}\alpha_u^{-1}\left(a^-+b\right)\right)\right)^- &\mbox{by \cref{def_matched}}\\
     &=\left(\alpha^{-1}_{\beta^{-1}_b\left(u\right)}\left(b^- \circ a  \circ a^- \circ \left(a^-+b\right)\right)\right)^-\\
&=\left(\alpha^{-1}_{\beta^{-1}_b\left(u\right)}\left(\rho_b(a)^-\right)\right)^-\\
    &=\left(\alpha^{-1}_{\beta_{\rho_b\left(a\right)}^{-1}\beta_{\rho_b\left(a\right)}\beta^{-1}_b\left(u\right)}\left(\rho_b(a)^-\right)\right)^- \\
    &=\alpha^{-1}_{\beta_{\rho_b\left(a\right)}\beta^{-1}_b\left(u\right)}\rho_b\left(a\right) &\mbox{by \eqref{inv}}
\end{align*}  
for all $(a, u), (b, v) \in S \times T$, thus also \eqref{eq:quinto} is satisfied.\\
Now, we prove that the maps $r_{S \ \bowtie \ T}$ and $r_S \bowtie r_T$ coincide. We observe that the first component of $r_{S \bowtie T}$ is equal to that of $r_S \bowtie r_T$ in \eqref{matchsolu} by \eqref{eq:terzo}-\eqref{eq:quarto}.\\
To complete the proof, it remains to show that
\begin{align*}
    \rho_{(b, v)}(a, u)=\left(\alphaa{-1}{\overline{U}}{\rhoo{\alphaa{}{\bar{u}}{\left(b\right)}}{\left(a\right)}},\,  \beta^{-1}_{\overline{A}}\rhoo{\beta_{\bar{a}}\left(v\right)}{\left(u\right)}\right),
\end{align*}
for all $(a, u), (b, v) \in S \times T$, where
\begin{center}
		   $\bar{a}:=\alphaa{-1}{u}{\left(a\right)}$, \,\,$\bar{u}:= \beta^{-1}_{a}\left(u\right)$,\,\, $A:=\alphaa{}{u}{\sigmaa{\bar{a}}{\left(b\right)}}$,\, $U:=\beta_a\sigmaa{\bar{u}}{\left(v\right)}$,\,\, $\overline{A}:=\alphaa{-1}{U}{\left(A\right)}$,\,\, $\overline{U}:= \beta^{-1}_{A}\left(U\right)$.
		\end{center}
Let $(a, u), (b, v) \in S \times T$. We compute
\begin{align*}
    \rho_{(b, v)}(a, u)&=\left(\alpha_u\sigma_{\bar{a}}(b), \, \beta_a\sigma_{\bar{u}}(v)\right)^- \circ  (a, u) \circ (b, v)\\
   &=\left(\alpha^{-1}_{\overline{U}}\left(\alpha_u\sigma_{\bar{a}}(b)^-\right), \, \beta^{-1}_{\overline{A}}\left(\beta_a\sigma_{\bar{u}}(v)^-\right)\right)\circ  \left(a \circ \alpha_{\bar{u}}(b),\ u \circ \beta_{\bar{a}}(v)\right).
\end{align*}
The first component of the latter product is equal to 
\begin{align*}
   \alpha^{-1}_{\overline{U}}\left(\alpha_u\sigma_{\bar{a}}(b)^-\right) \circ\, &\,\alpha_{\beta^{-1}_{\alpha^{-1}_{\overline{U}}\left(\alpha_u\sigma_{\bar{a}}(b)^-\right)}\left(\beta^{-1}_{\overline{A}}\left(\beta_a\sigma_{\bar{u}}(v)^-\right)\right)}\left(a \circ \alpha_{\bar{u}}(b)\right)\\
   &=\alpha^{-1}_{\overline{U}}(A^-)\circ \alpha_{\beta^{-1}_{\alpha^{-1}_{\overline{U}}\left(A^-\right)}\left(\beta^{-1}_{\overline{A}}\left(U^-\right)\right)}\left(a \circ \alpha_{\bar{u}}(b)\right).
\end{align*}
It suffices to compare the first components, as the second ones coincide by symmetry after interchanging the roles of $\alpha$ and $\beta$, $a$ and $u$, and $b$ and $v$. We have that
\[
\beta^{-1}_{
\alpha_{\overline U}^{-1}(A^{-})}
\beta_{\overline A}^{-1}(U^{-})
=
\beta_{A^{-}}^{-1}\left(\overline U\right)^{-}.
\]
Indeed, by repeatedly applying \eqref{inv}, one has
\begin{align*}
    \beta^{-1}_{A^-}(\overline{U})^-&=\beta_{A^-}^{-1}\left(\beta^{-1}_A(U)\right)^-=\beta^{-1}_{
\alpha_{\beta_A^{-1}(U)}^{-1}(A^-)}
\left(
\beta_A^{-1}(U)^{-}
\right) \\
&=
\beta^{-1}
_{\alpha_{\beta_A^{-1}(U)}^{-1}(A^{-})}
\left(
\beta_{\alpha_U^{-1}(A)}^{-1}(U^{-})
\right)=\beta^{-1}_{
\alpha_{\overline U}^{-1}(A^{-})}
\beta_{\overline A}^{-1}(U^{-}).
\end{align*}
Hence, the first component of $\rho_{(b, v)}(a, u)$ is equal to
\begin{align*}
   \alpha^{-1}_{\overline{U}}(A^-)\circ \alpha_{\beta^{-1}_{A^-}(\overline{U})^-}\left(a \circ \alpha_{\bar{u}}(b)\right).
\end{align*}
Finally, \begin{align*}
\alpha_{\bar U}^{-1}\rho_{\alpha_{\bar u}(b)}(a)
&=
\alpha_{\overline U}^{-1}\!\left(
A^{-}
\circ a
\circ \alpha_{\bar u}(b)
\right)                                                   \\
&=
\alpha_{\overline U}^{-1}\!\left(
A^{-}\circ
\alpha_{\beta_{A^{-}}^{-1}(\overline U)}\,
\alpha_{\beta_{A^{-}}^{-1}(\overline U)}^{-1}
\left(
a\circ\alpha_{\bar u}(b)
\right)
\right)                                                   \\
&=
\alpha_{\overline U}^{-1}(A^{-})
\circ
\alpha_{\beta_{A^{-}}^{-1}\left(\overline U\right)}^{-1}
\left(
a\circ\alpha_{\bar u}(b)
\right) &\mbox{by \cref{def_matched}}\\
&=\alpha^{-1}_{\overline{U}}(A^-)\circ \alpha_{\beta^{-1}_{A^-}(\overline{U})^-}\left(a \circ \alpha_{\bar{u}}(b)\right),
\end{align*}
thus the two components are equal. Therefore, the claim is proved.
   \end{proof}

\bigskip
Next, we focus on a particular case of the previous construction, namely the \emph{semidirect product} of two inverse braces, following a similar approach as in \cite[p. 597]{CaMaSt21}. \\Given a matched product system $\left(S, T, \alpha, \beta\right)$ of inverse braces, we consider $\beta_a = \id_T$, for all $a\in S$. In this way, the inverse semigroup $S\times T$ defined as in \cref{18} is exactly the semidirect product of the inverse semigroups $S$ and $T$ via $\alpha$, cf. \cite{Ni83} and \cite{Pr86}. 
Note that this is a particular case of the \emph{Zappa product} of two semigroups \cite{Ku83}. Let us recall this construction. In general, if $S$ and $T$ are semigroups, $\gamma:T\to S^{S}$ and $\delta:S\to T^{T}$ maps, set $^u a:= \gamma\left(u\right)\left(a\right)$ and $u^a:= \delta\left(a\right)\left(u\right)$, for all $a\in S$ and $u\in T$, if the following conditions are satisfied
	\begin{align}\label{S1}
	&^{u}(ab)=\,^ua \, ^{u^a}b	
	&^{uv}a = \, ^u(^va) \tag{Z1}\\
	\label{S2}&( uv)^a = u^{^va} \,\, v^a&u^{ab}=(u^a)^b\tag{Z2}
				\end{align}
for all $a,b \in S$ and $u,v \in T$, then $S \times T$ is a semigroup with respect to the operation defined by
	\begin{equation*}\label{prod-semigruppo-match}	\left(a,u\right)\left(b,v\right)=\left(a\,^ub, u^b\,v\right),
	\end{equation*}
	for all $a,b \in S$ and $u,v \in T$.
Thus, the inverse semigroup $(S\times T, \circ)$ defined as in \cref{18} is the Zappa product of $(S, \circ)$ and $(T, \circ)$ with $^{u}a = \gamma\left(u\right)\left(a\right) = \alpha_u\left(a\right)$  and $u^a = \beta_a\left(u\right) = u$, for all $a\in S$ and $u\in T$ (see \cite[Remark 5]{CaMaSt21}). 

\smallskip

We now state a characterization of when the semidirect product of two semigroups is an inverse semigroup.
\begin{theor}\label{th:char-semiprod} \emph{(\cite[cf. Theorem 6]{Pr86})}
	Let $S$ and $T$ be two semigroups and $\gamma:T\to \End(S)$ a homomorphism. Then, the semidirect product of $S$ and $T$ via $\sigma$ is an inverse semigroup if and only if:
	\begin{enumerate}
		\item[\rm{(1)}]$S$ and $T$ are inverse semigroups;
	\item[\rm{(2)}]$\gamma(T)\subseteq \Aut(S)$.
	\end{enumerate}
\end{theor}

\smallskip

We can now specialize \cref{teo_match} as follows.\begin{cor}\label{cor:semidirect}
	Let $S$ and $T$ be inverse braces and $\gamma:T\to \Aut(S)$ a homomorphism from $(T, \circ)$ into the automorphism group of the inverse brace $S$. Then, $S \times T$ equipped with the following operations
	\begin{align*}
	(a,u)+(b,v)&:=(a + b,\,u + v )\\
	(a,u)\circ(b,v)&:= (a\circ \gamma_u(b),\, u \circ v ),
	\end{align*}
	for all $(a,u), (b,v)\in S\times T$, is an inverse brace, called \emph{semidirect product of $S$ and $T$ via $\gamma$}, and denoted by $S\rtimes_{\gamma}T$.
\end{cor}
\smallskip

\begin{ex}
Let $S=\{1,x,y\}$ be a set and define the operation $\cdot$ as
$$\begin{tabular}{c | c c c c }
    $\cdot$ & $1$ & $x$ & $y$  \\
    \hline
   $1$& $1$ & $1$ & $1$ \\
    $x$ & $1$ & $x$ & $1$ \\
    $y$& $1$ & $1$ & $y$ 
\end{tabular}$$
that is, $(S, \cdot)$ is the upper semilattice with join $1$. Then $\Aut(S)=\{\id_S, \tau\}$, where $\tau := (x\, y)$. Let $(S, \cdot, \cdot)$ be the trivial weak brace on $S$. \\
Moreover, let $T=\{a,b,c,d\}$ and consider the inverse brace $(T, +, \circ)$ in \cref{ex-4ele} given by
    $$\begin{tabular}{c | c c c c c  c}
    $+$ & $a$ & $b$ & $c$  &$d$ \\
    \hline
   $a$& $a$ & $a$ & $a$ &$d$\\
    $b$ & $a$ & $b$ & $c$ & $d$\\
    $c$& $a$ & $c$ & $b$ & $d$\\
    $d$ & $d$& $d$& $d$& $a$
\end{tabular}  \qquad \qquad{\rm and}\qquad \qquad \begin{tabular}{c | c c c c c  c}
    $\circ$ & $a$ & $b$ & $c$ &$d$  \\
    \hline
   $a$& $a$ & $a$ & $a$ & $d$\\
    $b$ & $a$ & $b$ & $c$ & $d$\\
    $c$& $a$ & $c$ & $c$ & $d$\\
    $d$& $d$ & $d$ & $d$ & $a$
\end{tabular}$$
Then it is a routine computation to check that the map $\gamma:T\to \Aut(S)$ defined by $$\gamma(a) = \gamma(b) = \gamma(c)= \id_S \quad \text{and} \quad \gamma(d) = \tau$$
is a homomorphism from $(T, \circ)$ into $\Aut(S)$. 
Hence, by \cref{cor:semidirect}, 
$S\rtimes_{\gamma}T$ is a new inverse brace. 
\end{ex}

\medskip

\subsection{Strong semilattice of inverse braces}

\smallskip

We present a method for constructing inverse braces using the technique of strong semilattices of inverse braces. This technique has also been successfully applied to other brace-like structures and has proven particularly effective for dual weak braces, which can be shown to be strong semilattices of skew braces \cite[Theorem 1]{CaMaSt24}. The description of Clifford semigroups inspires this technique as they are \emph{strong semilattices of groups} (see \cite[Theorem II.2.6]{Pe84}). To make this paper self-contained, let us initially recall this description.

Let $Y$ be a (lower) semilattice and $\{G_{\alpha}\ | \ \alpha\in Y\}$ a family of disjoint groups. For all $\alpha,\beta\in Y$ such that $\alpha \geq \beta$, let $\phii{\alpha}{\beta}:G_{\alpha}\to G_{\beta}$ be a homomorphism of groups such that
 \begin{enumerate}
     \item $\phii{\alpha}{\alpha}$ is the identical automorphism of $G_{\alpha}$, for every $\alpha \in Y$;
     \item $\phii{\beta}{\gamma}{}\phii{\alpha}{\beta}{} = \phii{\alpha}{\gamma}{}$, for all $\alpha, \beta, \gamma \in Y$ such that $\alpha \geq \beta \geq \gamma$.
 \end{enumerate} 
Then, $S = \displaystyle \bigcup_{\alpha\in Y}\,G_{\alpha}$ endowed with the operation defined by
$$
    ab:= \phii{\alpha}{\alpha\beta}(a)\ \phii{\beta}{\alpha\beta}(b),
$$
for all $a\in G_{\alpha}$ and $b\in G_{\beta}$, is a \emph{Clifford semigroup}, also called \emph{strong semilattice $Y$ of groups $G_{\alpha}$}, usually written as $S=[Y; G_{\alpha}; \phii{\alpha}{\beta}]$. Conversely, any Clifford semigroup is of this form.

\smallskip

We can give the following result.
\begin{theor}\label{teo_semila}
Let $Y$ be a (lower) semilattice, $\left\{S_{\alpha}\ \left|\ \alpha \in Y\right.\right\}$ a family of disjoint inverse braces. For each pair $\alpha,\beta$ of elements of $Y$ such that $\alpha \geq \beta$, let $\phii{\alpha}{\beta}:S_{\alpha}\to S_{\beta}$ be an inverse brace homomorphism such that
\begin{enumerate}
    \item $\phii{\alpha}{\alpha}=\id_{S_{\alpha}}$, for every $\alpha \in Y$,
    \item $\phii{\beta}{\gamma}{}\phii{\alpha}{\beta}{} = \phii{\alpha}{\gamma}{}$, for all $\alpha, \beta, \gamma \in Y$ such that $\alpha \geq \beta \geq \gamma$.
\end{enumerate}
Then, $S = \bigcup\left\{S_{\alpha}\ \left|\ \alpha\in Y\right.\right\}$ endowed with the addition and the multiplication, respectively, defined by
\begin{align*}
    a+b:= \phii{\alpha}{\alpha\beta}(a)+\phii{\beta}{\alpha\beta}(b)
    \quad\text{and}\quad
     a\circ b:= \phii{\alpha}{\alpha\beta}(a)\circ\phii{\beta}{\alpha\beta}(b),
\end{align*}
for all $a\in S_{\alpha}$ and $b\in S_{\beta}$, is an inverse brace, called \emph{strong semilattice of inverse braces} and
denoted by $S=[Y; S_\alpha;\phii{\alpha}{\beta}]$. \end{theor}
\begin{proof}
Initially, $(S, +)$ and $(S, \circ)$ are inverse semigroups as a consequence of \cite[Ex.(ii), p.90]{Pe84}. Moreover, if $\alpha, \beta, \gamma \in Y$ and $a \in S_\alpha$, $b \in S_\beta$, and $c \in S_\gamma$, we obtain
\begin{align*}
    a \circ (b+c)
    &=\phii{\alpha}{\alpha\beta\gamma}(a) \circ \phii{\beta\gamma}{\alpha\beta\gamma}\left(\phii{\beta}{\beta\gamma}(b)+\phii{\gamma}{\beta\gamma}(c)\right)\\
    &=\phii{\alpha}{\alpha\beta\gamma}(a) \circ \left(\phii{\beta}{\alpha\beta\gamma}(b)+\phii{\gamma}{\alpha\beta\gamma}(c)\right)\\
    &=\phii{\alpha}{\alpha\beta\gamma}(a) \circ \phii{\beta}{\alpha\beta\gamma}(b)-\phii{\alpha}{\alpha\beta\gamma}(a)+\phii{\alpha}{\alpha\beta\gamma}(a) \circ  \phii{\gamma}{\alpha\beta\gamma}(c)\\
    &=\phii{\alpha\beta}{\alpha\beta\gamma}\left(\phii{\alpha}{\alpha\beta}(a) \circ \phii{\beta}{\alpha\beta}(b)\right)-\phii{\alpha\beta\gamma}{\alpha\beta\gamma}\phii{\alpha}{\alpha\beta\gamma}(a)+\phii{\alpha}{\alpha\beta\gamma}(a) \circ \phii{\gamma}{\alpha\beta\gamma}(c)\\
    &=\phii{\alpha}{\alpha\beta}(a) \circ \phii{\beta}{\alpha\beta}(b)-\phii{\alpha}{\alpha\beta\gamma}(a)+\phii{\alpha}{\alpha\beta\gamma}(a) \circ \phii{\gamma}{\alpha\beta\gamma}(c)\\
    &=\phii{\alpha}{\alpha\beta}(a) \circ \phii{\beta}{\alpha\beta}(b)-\phii{\alpha\beta}{\alpha\beta\gamma}\phii{\alpha}{\alpha\beta}(a)+\phii{\alpha\gamma}{\alpha\beta\gamma}\left(\phii{\alpha}{\alpha\gamma}(a) \circ \phii{\gamma}{\alpha\gamma}(c)\right)\\
    &=\phii{\alpha\beta}{\alpha\beta}\left(\phii{\alpha}{\alpha\beta}(a) \circ \phii{\beta}{\alpha\beta}(b)\right)-\phii{\alpha}{\alpha\beta}(a)+\phii{\alpha}{\alpha\gamma}(a) \circ \phii{\gamma}{\alpha\gamma}(c)\\
    &=\phii{\alpha}{\alpha\beta}(a) \circ \phii{\beta}{\alpha\beta}(b) - a+\phii{\alpha}{\alpha\gamma}(a) \circ \phii{\gamma}{\alpha\gamma}(c) \\
    &=a \circ b -a+a \circ c.
\end{align*}
Then $(S, +, \circ)$ is an inverse brace.
\end{proof}

\smallskip

\begin{theor}\label{teo_solu_semil}
 Let  $S=[Y; S_\alpha;\phii{\alpha}{\beta}]$ be a strong semilattice of inverse braces, each of which satisfies \eqref{A} and \eqref{B}. Then the map $r_S:S \times S \to S  \times S$ defined as in \eqref{r_S} is a solution. Moreover, $r_S$ is the strong semilattice of the solutions $r_\alpha$ associated with each inverse brace $S_\alpha$, namely,
     	\begin{align}\label{solu_semil}
		r_S\left(a, b\right)= 
		r_{\alpha\beta}\left(\phi_{\alpha,\alpha\beta}\left(a\right),
		\phi_{\beta, \alpha\beta}\left(b\right) \right),
	\end{align}	
    for all $a\in S_{\alpha}$ and $b\in S_{\beta}$.
    \begin{proof}
        Initially, we show that the equations \eqref{A} and \eqref{B} are valid in $S$. Let $a \in S_\alpha$ and $e \in \E\left(S, \circ \right)$. Then there exists $\beta \in Y$ such that $e \in \E\left(S_\beta, \circ \right)$. Hence,  we have
        \begin{align*}
            \sigma_a(e)&=\phii{\alpha}{\alpha\beta}(a) \circ \left(\phii{\alpha}{\alpha\beta}(a^-)+\phii{\beta}{\alpha\beta}(e)\right)=\sigma^{[\alpha\beta]}_{\phii{\alpha}{\alpha\beta}(a)}\left(\phii{\beta}{\alpha\beta}(e)\right)\\
            &=\phii{\alpha}{\alpha\beta}(a)  \circ \phii{\beta}{\alpha\beta}(e) \circ \phii{\alpha}{\alpha\beta}(a^-) =a \circ e \circ a^-,
        \end{align*}
       where we use \eqref{A} for the inverse brace $S_{\alpha\beta}$. Moreover, if $b \in S_\beta$, we get
        \begin{align*}
           \sigma_a\left( b \circ b^-\right)&=\phii{\alpha}{\alpha\beta}(a) \circ \left(\phii{\alpha}{\alpha\beta}(a^-)+\phii{\beta}{\alpha\beta}\left(b \circ b^-\right)\right)=\sigma^{[\alpha\beta]}_{\phii{\alpha}{\alpha\beta}(a)}\left(\phii{\beta}{\alpha\beta}\left(b \circ b^-\right)\right)\\
           &=\sigma^{[\alpha\beta]}_{\phii{\alpha}{\alpha\beta}(a)}\left(\phii{\beta}{\alpha\beta}(b)\right) \circ \sigma^{[\alpha\beta]}_{\phii{\alpha}{\alpha\beta}(a)}\left(\phii{\beta}{\alpha\beta}(b)\right)^-=\sigma_a(b) \circ \sigma_a(b)^-,
        \end{align*}
       where, again, we use \eqref{B} for the inverse brace $S_{\alpha\beta}$.
        Thus, by \cref{teo_solu}, the map $r_S$ is a solution. Furthermore, by \cite[Theorem 4.1]{CCoSt20x-2}, to get the second claim, we have to show that
        \begin{align*}
            \left(\phi_{\alpha,\beta}\times \phi_{\alpha,\beta}\right)r_{\alpha}
		= r_{\beta}\left(\phi_{\alpha, \beta}\times \phi_{\alpha, \beta}\right), 
        \end{align*}
for all $\alpha, \beta\in Y$ such that  $\alpha \geq \beta$. This follows by
\begin{align*}
\left(\phi_{\alpha,\beta}\times \phi_{\alpha,\beta}\right)r_{\alpha}(a, a')&=  \left(\phi_{\alpha,\beta}\left(\sigma^{[\alpha]}_a\left(a'\right)\right), \phi_{\alpha,\beta}\left(\rho^{[\alpha]}_{a'}\left(a\right)\right)\right)\\
 &=\left(\sigma^{[\beta]}_{\phii{\alpha}{\beta}(a)}\left(\phii{\alpha}{\beta}\left(a'\right)\right), \rho^{[\beta]}_{\phii{\alpha}{\beta}\left(a'\right)}\left(\phii{\alpha}{\beta}\left(a\right)\right)\right),
\end{align*}
for all $a, a' \in S_\alpha$, since $\phi_{\alpha,\beta}$ is an inverse brace homomorphism.
    \end{proof}
\end{theor}

\smallskip

\begin{ex}
Let $S=\{0,x,y\}$ and $\mathbb{Z}_2=\{a,b\}$. Consider the inverse braces $(S, +, \circ)$ and $\left(\mathbb{Z}_2, +, \circ\right)$ in \cref{ex_cliff}  and in \cref{esempiowang}, respectively. Moreover, consider $Y=\{\gamma, \delta \}$ with $\gamma \geq \delta$ and set $S_\gamma:=S$ and $S_\delta:=\mathbb{Z}_2$. Then the map $\phi_{\gamma, \delta } : S_\gamma \to S_\delta$, given by
    \begin{align*}
       \phi_{\gamma, \delta }(0)=\phi_{\gamma, \delta }(x)=a, \quad \phi_{\gamma, \delta }(y)=b
    \end{align*}
    is a homomorphism from the inverse brace $S_\alpha$ into the inverse brace $S_\beta$. Hence, by \cref{teo_semila}, set $X=S_\alpha \cup S_\beta$, we have that $X:=[Y; S_{\alpha}, \phi_{\alpha, \beta}]$ is an inverse brace, whose operations are the following
       $$\begin{tabular}{c | c c c c c}
    $+$ & $0$ & $x$ & $y$  &$a$ &$b$\\
    \hline
$0$& $0$ & $x$ & $y$  &$a$ &$b$\\
    $x$ & $x$ & $x$ & $y$ & $a$  &$b$\\
    $y$& $y$ & $y$ & $x$ & $b$  &$a$\\
    $a$ & $a$& $a$& $b$& $a$  &$b$\\
    $b$& $b$& $b$& $a$& $b$  &$a$
\end{tabular} 
\qquad \quad{\rm and}\qquad \quad 
\begin{tabular}{c | c c c c c  c c}
    $\circ$ & $0$ & $x$ & $y$  &$a$ &$b$  \\
    \hline
$0$& $0$ & $x$ & $y$  &$a$ &$b$\\
    $x$ & $x$ & $x$ & $y$ & $a$  &$b$\\
    $y$& $y$ & $y$ & $y$ & $b$  &$b$\\
    $a$ & $a$& $a$& $b$& $a$  &$b$\\
    $b$& $b$& $b$& $b$& $b$  &$b$
\end{tabular}
$$
    In addition, by \cref{teo_solu_semil}, the map $r:X \times X \to X \times X$ defined as in \eqref{solu_semil} is a solution on $X$. 
\end{ex}

\medskip

\section{Final remarks and further directions}
In this brief section, we present some open problems and discuss possible research directions.

\smallskip

First, it is natural to remark that there exist inverse braces for which the map $\sigma$ is not a homomorphism. The inverse brace in \cref{nhomo} is such an example. The following provides another example.

\begin{ex} 
 Let $S=\{0, a,b, c, d\}$. Consider the inverse monoids $(S, +)$ and $(S, \circ)$ both with identity $0$ and binary operations given by 
  $$\begin{tabular}{c | c c c c c}
    $+$ & $0$ & $a$ & $b$ & $c$ & $d$ \\
    \hline
    $0$ & $0$ & $a$ & $b$ & $c$ & $d$ \\
    $a$  & $a$ & $a$ & $b$ & $c$ & $d$ \\
    $b$ & $b$ & $b$ & $a$ & $d$ & $c$ \\
    $c$ & $c$ & $c$ & $d$ & $a$ & $b$\\
    $d$ & $d$ & $d$ & $c$ & $b$ & $a$
\end{tabular} \qquad \text{and} \qquad 
\begin{tabular}{c | c c c c c } $\circ$ & $0$ & $a$ & $b$ & $c$ & $d$ 
\\ \hline
    $0$ & $0$ & $a$ & $b$ & $c$ & $d$ \\
    $a$  & $a$ & $a$ & $b$ & $c$ & $d$ \\
    $b$ & $b$ & $b$ & $b$ & $b$ &  $b$  \\
    $c$ & $c$ & $c$ & $b$ & $d$ & $a$\\
    $d$ & $d$ & $d$ & $b$ & $a$ & $c$
\end{tabular}
$$
Then $(S,+,\circ)$ is an inverse brace that does not satisfy \eqref{A}. 
Indeed, for instance, 
\[d\circ b \circ b^- \circ d^- = b
\quad \text{and} \quad 
\sigma_d\left(b \circ b^-\right)
=c.\]
\end{ex}
In all the examples considered, including computations performed with GAP up to order $5$, a map of the form in \eqref{r_S}, or of the form in \cref{teo_solu_weak} involving the $\lambda$ map, is not a solution.
These structures may be studied independently, but it is currently unknown whether they give rise to solutions written in a different form.

\medskip

Another open question is whether $(S, +)$ is always a Clifford semigroup or whether suitable additional assumptions force it to be a Clifford semigroup. For weak braces, this property necessarily holds (see \cite[Theorem 8]{CaMaMiSt22}). In the proof, the additional condition on idempotents makes the argument easier, suggesting why the result holds in this setting.

A computation performed with GAP shows that, for orders up to $5$, there are no inverse braces whose additive structure is a Brandt semigroup (cf. \cite[p. 152]{Ho95}), the only inverse semigroup of order $5$ that is not Clifford, up to isomorphism.

 \medskip

Finally, another possible research direction would be to investigate the existence of \emph{deformed solutions} associated with inverse braces. Following the approach of \cite{DoRy24, MaRySt25}, it is natural to ask whether these structures also admit a family of parameter-dependent solutions, where the parameter belongs to a distributive-like structure. In the specific, by \cite[Theorem 3.9]{MaRySt25}, in the case of a dual weak brace $S$, given an element $z \in S$, the map $r_{z}:S \times S \to S \times S$ defined by
\begin{align*}
       r_{z}(a,b)=\left(-a\circ z+a\circ b\circ z, \left(-a\circ z+a\circ b\circ z\right)^{-} \circ a \circ b \right),
    \end{align*}
    for all $a, b \in S$, is a solution if and only if $z \in \mathcal{D}_r(S)$, where
\begin{align*}
    \mathcal{D}_r(S)=\{z \in S \, \mid \, \forall \, a,b \in S \quad (a+b) \circ z=a\circ z-z+b \circ z\}.
\end{align*}
In particular, $\E(S) \subseteq \mathcal{D}_r(S)$, and hence each idempotent of $S$ gives rise to a deformation of the classical solution associated with $S$. This observation is significant because, in the case of skew braces, the only idempotent is the identity $0$ of the two groups, so no non-trivial deformation arises with $z=0$.

In the setting of inverse braces, the first step is to understand how to deform the solution given in \eqref{r_S}, and, if possible, to determine the conditions on the parameter and on the underlying structure under which the deformed map yields a solution.

\bigskip
\bibliography{bibliography}

@preamble{
   "\def\cprime{$'$} "
}

@incollection {Dr92,
    AUTHOR = {Drinfel\cprime d, V. G.},
     TITLE = {On some unsolved problems in quantum group theory},
 BOOKTITLE = {Quantum groups ({L}eningrad, 1990)},
    SERIES = {Lecture Notes in Math.},
    VOLUME = {1510},
     PAGES = {1--8},
 PUBLISHER = {Springer, Berlin},
      YEAR = {1992},
   MRCLASS = {17B37 (16W30 81R50)},
  MRNUMBER = {1183474},
MRREVIEWER = {Yvette Kosmann-Schwarzbach},
       DOI = {10.1007/BFb0101175},
       URL = {https://doi.org/10.1007/BFb0101175},
}

@article {JePi25,
    AUTHOR = {Jedlicka, Premysl and Pilitowska, Agata},
     TITLE = {Diagonals of solutions of the {Y}ang-{B}axter equation},
   JOURNAL = {Forum Math.},
  FJOURNAL = {Forum Mathematicum},
    VOLUME = {38},
      YEAR = {2026},
    NUMBER = {2},
     PAGES = {321--338},
      ISSN = {0933-7741,1435-5337},
   MRCLASS = {16T25 (08A05 20B30)},
  MRNUMBER = {5009092},
       DOI = {10.1515/forum-2024-0409},
       URL = {https://doi.org/10.1515/forum-2024-0409},
}

@article {GoWa26,
    AUTHOR = {Gong, Xiaoqian and Wang, Shoufeng},
     TITLE = {Post {C}lifford semigroups, the {Y}ang-{B}axter equation,
              relative {R}ota-{B}axter {C}lifford semigroups and dual weak
              left braces},
   JOURNAL = {Comm. Algebra},
  FJOURNAL = {Communications in Algebra},
    VOLUME = {54},
      YEAR = {2026},
    NUMBER = {1},
     PAGES = {249--273},
      ISSN = {0092-7872,1532-4125},
   MRCLASS = {16T25 (16Y99 17B38 20M18)},
  MRNUMBER = {4984587},
       DOI = {10.1080/00927872.2025.2522905},
       URL = {https://doi.org/10.1080/00927872.2025.2522905},
}

@article {LiWa25,
    AUTHOR = {Liu, Qianxue and Wang, Shoufeng},
     TITLE = {Set-theoretic solutions of the {Y}ang-{B}axter equation
              associated to weak left {$\star$}-braces},
   JOURNAL = {Semigroup Forum},
  FJOURNAL = {Semigroup Forum},
    VOLUME = {110},
      YEAR = {2025},
    NUMBER = {3},
     PAGES = {615--639},
      ISSN = {0037-1912,1432-2137},
   MRCLASS = {16T25 (20M17)},
  MRNUMBER = {4920048},
       DOI = {10.1007/s00233-025-10539-w},
       URL = {https://doi.org/10.1007/s00233-025-10539-w},
}

@article {DoRy24,
    AUTHOR = {Doikou, Anastasia and Rybolowicz, Bernard},
     TITLE = {Novel non-involutive solutions of the {Y}ang-{B}axter equation
              from (skew) braces},
   JOURNAL = {J. Lond. Math. Soc. (2)},
  FJOURNAL = {Journal of the London Mathematical Society. Second Series},
    VOLUME = {110},
      YEAR = {2024},
    NUMBER = {4},
     PAGES = {Paper No. e12999, 23},
      ISSN = {0024-6107,1469-7750},
   MRCLASS = {16T25},
  MRNUMBER = {4803940},
MRREVIEWER = {Marzia\ Mazzotta},
       DOI = {10.1112/jlms.12999},
       URL = {https://doi.org/10.1112/jlms.12999},
}

@book {Law26,
    AUTHOR = {Lawson, Mark V.},
     TITLE = {Inverse semigroups---the theory of partial symmetries},
   EDITION = {Second},
 PUBLISHER = {World Scientific Publishing Co. Pte. Ltd., Hackensack, NJ},
      YEAR = {[2026] \copyright 2026},
     PAGES = {xviii+459},
      ISBN = {[9789819816767]; [9789819816774]; [9789819816781]},
   MRCLASS = {20M18},
  MRNUMBER = {5000115},
}

@article{MaSteWi26x,
author = {Mazzotta, Marzia and Stefanelli, Paola and Wiertel, Magdalena},
title = {Quasi racks, quasi bijective and quasi non-degenerate set-theoretic solutions of the {Y}ang–{B}axter equation},
journal = {Journal of Algebra and Its Applications},
volume = {0},
number = {0},
pages = {2750198},
year = {0},
doi = {10.1142/S0219498827501982},

URL = { 
    
        https://doi.org/10.1142/S0219498827501982
    
    

},
}

@article {Xi88,
    AUTHOR = {Du, Xian Kun},
     TITLE = {The rings with regular adjoint semigroups},
   JOURNAL = {Northeast. Math. J.},
  FJOURNAL = {Northeastern Mathematical Journal. Dongbei Shuxue},
    VOLUME = {4},
      YEAR = {1988},
    NUMBER = {4},
     PAGES = {463--468},
      ISSN = {1000-1778},
   MRCLASS = {16A22 (16A30 20M17)},
  MRNUMBER = {987071},
MRREVIEWER = {Jan\ Okni\'nski},
}

@article {CaMaSt23,
    AUTHOR = {Catino, Francesco and Mazzotta, Marzia and Stefanelli, Paola},
     TITLE = {Rota-{B}axter operators on {C}lifford semigroups and the
              {Y}ang-{B}axter equation},
   JOURNAL = {J. Algebra},
  FJOURNAL = {Journal of Algebra},
    VOLUME = {622},
      YEAR = {2023},
     PAGES = {587--613},
      ISSN = {0021-8693,1090-266X},
   MRCLASS = {20M18 (16T25 81R50)},
  MRNUMBER = {4551911},
MRREVIEWER = {P.\ M.\ Higgins},
       DOI = {10.1016/j.jalgebra.2023.02.013},
       URL = {https://doi.org/10.1016/j.jalgebra.2023.02.013},
}

@misc{Wang25x,
      title={Some characterizations of weak left braces, \textnormal{preprint arXiv:2502.05861}}, 
      author={Shoufeng Wang},
      year={2025},
      eprint={2502.05861},
      archivePrefix={arXiv},
      primaryClass={math.GR},
      url={https://arxiv.org/abs/2502.05861}, 
}

@article {CaMaMiSt22,
    AUTHOR = {Catino, Francesco and Mazzotta, Marzia and Miccoli, Maria
              Maddalena and Stefanelli, Paola},
     TITLE = {Set-theoretic solutions of the {Y}ang-{B}axter equation
              associated to weak braces},
   JOURNAL = {Semigroup Forum},
  FJOURNAL = {Semigroup Forum},
    VOLUME = {104},
      YEAR = {2022},
    NUMBER = {2},
     PAGES = {228--255},
      ISSN = {0037-1912,1432-2137},
   MRCLASS = {16Y99 (16T99 20M18)},
  MRNUMBER = {4412793},
MRREVIEWER = {Bernard\ Rybo\l owicz},
       DOI = {10.1007/s00233-022-10264-8},
       URL = {https://doi.org/10.1007/s00233-022-10264-8},
}

@article {MaRySt25,
    AUTHOR = {Mazzotta, Marzia and Rybolowicz, Bernard and Stefanelli,
              Paola},
     TITLE = {Deformed solutions of the {Y}ang-{B}axter equation associated
              to dual weak braces},
   JOURNAL = {Ann. Mat. Pura Appl. (4)},
  FJOURNAL = {Annali di Matematica Pura ed Applicata. Series IV},
    VOLUME = {204},
      YEAR = {2025},
    NUMBER = {2},
     PAGES = {711--731},
      ISSN = {0373-3114,1618-1891},
   MRCLASS = {16T25 (16Y99 20M18 81R50)},
  MRNUMBER = {4883233},
       DOI = {10.1007/s10231-024-01502-7},
       URL = {https://doi.org/10.1007/s10231-024-01502-7},
}

@article {CaMaSt24,
    AUTHOR = {Catino, Francesco and Mazzotta, Marzia and Stefanelli, Paola},
     TITLE = {Solutions of the {Y}ang-{B}axter equation and strong
              semilattices of skew braces},
   JOURNAL = {Mediterr. J. Math.},
  FJOURNAL = {Mediterranean Journal of Mathematics},
    VOLUME = {21},
      YEAR = {2024},
    NUMBER = {2},
     PAGES = {Paper No. 67, 22},
      ISSN = {1660-5446,1660-5454},
   MRCLASS = {16T25 (16Y99 20M18 81R50)},
  MRNUMBER = {4718672},
MRREVIEWER = {Kun\ Hao},
       DOI = {10.1007/s00009-024-02611-6},
       URL = {https://doi.org/10.1007/s00009-024-02611-6},
}

@book {Ho95,
	AUTHOR = {Howie, John M.},
	TITLE = {Fundamentals of semigroup theory},
	SERIES = {London Mathematical Society Monographs. New Series},
	VOLUME = {12},
	NOTE = {{O}xford Science Publications},
	PUBLISHER = {The Clarendon Press, Oxford University Press, New York},
	YEAR = {1995},
	PAGES = {x+351},
	ISBN = {0-19-851194-9},
	MRCLASS = {20Mxx (20-02)},
	MRNUMBER = {1455373},
	MRREVIEWER = {P. M. Higgins},
}

@article {CeJeOk14,
    AUTHOR = {Ced\'{o}, Ferran and Jespers, Eric and Okni\'{n}ski, Jan},
     TITLE = {Braces and the {Y}ang-{B}axter equation},
   JOURNAL = {Comm. Math. Phys.},
  FJOURNAL = {Communications in Mathematical Physics},
    VOLUME = {327},
      YEAR = {2014},
    NUMBER = {1},
     PAGES = {101--116},
      ISSN = {0010-3616},
   MRCLASS = {81Q05 (16T25)},
  MRNUMBER = {3177933},
MRREVIEWER = {Gangcheng Wang},
       DOI = {10.1007/s00220-014-1935-y},
       URL = {https://doi.org/10.1007/s00220-014-1935-y},
}

@article {Ku83,
    AUTHOR = {Kunze, M.},
     TITLE = {Zappa products},
   JOURNAL = {Acta Math. Hungar.},
  FJOURNAL = {Acta Mathematica Hungarica},
    VOLUME = {41},
      YEAR = {1983},
    NUMBER = {3-4},
     PAGES = {225--239},
      ISSN = {0236-5294},
   MRCLASS = {20M10 (20M35)},
  MRNUMBER = {703736},
MRREVIEWER = {Heinrich Seidel},
       DOI = {10.1007/BF01961311},
       URL = {https://doi.org/10.1007/BF01961311},
}

@article {Ru07,
    AUTHOR = {Rump, Wolfgang},
     TITLE = {Braces, radical rings, and the quantum {Y}ang-{B}axter
              equation},
   JOURNAL = {J. Algebra},
  FJOURNAL = {Journal of Algebra},
    VOLUME = {307},
      YEAR = {2007},
    NUMBER = {1},
     PAGES = {153--170},
      ISSN = {0021-8693},
   MRCLASS = {16Y99 (16W30)},
  MRNUMBER = {2278047},
MRREVIEWER = {Gigel Militaru},
       DOI = {10.1016/j.jalgebra.2006.03.040},
       URL = {https://doi.org/10.1016/j.jalgebra.2006.03.040},
}

@article {GuVe17,
    AUTHOR = {Guarnieri, L. and Vendramin, L.},
     TITLE = {Skew braces and the {Y}ang-{B}axter equation},
   JOURNAL = {Math. Comp.},
  FJOURNAL = {Mathematics of Computation},
    VOLUME = {86},
      YEAR = {2017},
    NUMBER = {307},
     PAGES = {2519--2534},
      ISSN = {0025-5718},
   MRCLASS = {16T25 (81R50)},
  MRNUMBER = {3647970},
MRREVIEWER = {Paola Stefanelli},
       DOI = {10.1090/mcom/3161},
       URL = {https://doi.org/10.1090/mcom/3161},
}

@article {Ni83,
    AUTHOR = {Nico, William R.},
     TITLE = {On the regularity of semidirect products},
   JOURNAL = {J. Algebra},
  FJOURNAL = {Journal of Algebra},
    VOLUME = {80},
      YEAR = {1983},
    NUMBER = {1},
     PAGES = {29--36},
      ISSN = {0021-8693},
   MRCLASS = {20M10},
  MRNUMBER = {690701},
MRREVIEWER = {Yu. G. Koshelev},
       DOI = {10.1016/0021-8693(83)90015-7},
       URL = {https://doi.org/10.1016/0021-8693(83)90015-7},
}

@article{Pr86,
    AUTHOR = {Preston, G. B.},
     TITLE = {Semidirect products of semigroups},
   JOURNAL = {Proc. Roy. Soc. Edinburgh Sect. A},
  FJOURNAL = {Proceedings of the Royal Society of Edinburgh. Section A.
              Mathematics},
    VOLUME = {102},
      YEAR = {1986},
    NUMBER = {1-2},
     PAGES = {91--102},
      ISSN = {0308-2105},
   MRCLASS = {20M10},
  MRNUMBER = {837162},
MRREVIEWER = {John M. Howie},
       DOI = {10.1017/S0308210500014505},
       URL = {https://doi.org/10.1017/S0308210500014505},
}

@article {CCoSt20,
    AUTHOR = {Catino, Francesco and Colazzo, Ilaria and Stefanelli, Paola},
     TITLE = {The matched product of set-theoretical solutions of the
              {Y}ang-{B}axter equation},
   JOURNAL = {J. Pure Appl. Algebra},
  FJOURNAL = {Journal of Pure and Applied Algebra},
    VOLUME = {224},
      YEAR = {2020},
    NUMBER = {3},
     PAGES = {1173--1194},
      ISSN = {0022-4049},
   MRCLASS = {16T25 (16N20 16Y99 81R50)},
  MRNUMBER = {4009573},
       DOI = {10.1016/j.jpaa.2019.07.012},
       URL = {https://doi.org/10.1016/j.jpaa.2019.07.012},
}

@article {CCoSt20x-2,
    AUTHOR = {Catino, Francesco and Colazzo, Ilaria and Stefanelli, Paola},
     TITLE = {Set-theoretic solutions to the {Y}ang--{B}axter equation and
              generalized semi-braces},
   JOURNAL = {Forum Math.},
  FJOURNAL = {Forum Mathematicum},
    VOLUME = {33},
      YEAR = {2021},
    NUMBER = {3},
     PAGES = {757--772},
      ISSN = {0933-7741},
   MRCLASS = {16T25 (16N20 16Y99 81R50)},
  MRNUMBER = {4250483},
       DOI = {10.1515/forum-2020-0082},
       URL = {https://doi.org/10.1515/forum-2020-0082},
}

@book{Pe84,
    AUTHOR = {Petrich, Mario},
     TITLE = {Inverse semigroups},
    SERIES = {Pure and Applied Mathematics (New York)},
      NOTE = {A Wiley-Interscience Publication},
 PUBLISHER = {John Wiley \& Sons, Inc., New York},
      YEAR = {1984},
     PAGES = {x+674},
      ISBN = {0-471-87545-7},
   MRCLASS = {20-02 (20Mxx)},
  MRNUMBER = {752899},
MRREVIEWER = {N. R. Reilly},
}

@article {Ya67,
    AUTHOR = {Yang, C. N.},
     TITLE = {Some exact results for the many-body problem in one dimension
              with repulsive delta-function interaction},
   JOURNAL = {Phys. Rev. Lett.},
  FJOURNAL = {Physical Review Letters},
    VOLUME = {19},
      YEAR = {1967},
     PAGES = {1312--1315},
      ISSN = {0031-9007},
   MRCLASS = {81.20},
  MRNUMBER = {261870},
MRREVIEWER = {S. Deser},
       DOI = {10.1103/PhysRevLett.19.1312},
       URL = {https://doi.org/10.1103/PhysRevLett.19.1312},
}

@article {Ba72,
    AUTHOR = {Baxter, Rodney J.},
     TITLE = {Partition function of the eight-vertex lattice model},
   JOURNAL = {Ann. Physics},
  FJOURNAL = {Annals of Physics},
    VOLUME = {70},
      YEAR = {1972},
     PAGES = {193--228},
      ISSN = {0003-4916},
   MRCLASS = {82.46},
  MRNUMBER = {290733},
MRREVIEWER = {S. Sherman},
       DOI = {10.1016/0003-4916(72)90335-1},
       URL = {https://doi.org/10.1016/0003-4916(72)90335-1},
}

@article {CaMaSt21,
    AUTHOR = {Catino, Francesco and Mazzotta, Marzia and Stefanelli, Paola},
     TITLE = {Inverse semi-braces and the {Y}ang-{B}axter equation},
   JOURNAL = {J. Algebra},
  FJOURNAL = {Journal of Algebra},
    VOLUME = {573},
      YEAR = {2021},
     PAGES = {576--619},
      ISSN = {0021-8693},
       DOI = {10.1016/j.jalgebra.2021.01.009},
       URL = {https://doi.org/10.1016/j.jalgebra.2021.01.009},
}

@book{La98,
    AUTHOR = {Lawson, Mark V.},
     TITLE = {Inverse semigroups},
      NOTE = {{T}he theory of partial symmetries},
 PUBLISHER = {World Scientific Publishing Co., Inc., River Edge, NJ},
      YEAR = {1998},
     PAGES = {xiv+411},
      ISBN = {981-02-3316-7},
   MRCLASS = {20M18 (18B40)},
  MRNUMBER = {1694900},
MRREVIEWER = {Peter R. Jones},
       DOI = {10.1142/9789812816689},
       URL = {https://doi.org/10.1142/9789812816689},
}

\bigskip

\end{document}